\documentclass[journal,twoside,web]{ieeecolor}

\usepackage{bbm}
\usepackage{mathrsfs}
\usepackage{generic}
\usepackage{cite}
\usepackage{amsmath,amssymb,amsfonts}
\usepackage[hidelinks]{hyperref}
\usepackage[algo2e,ruled]{algorithm2e}
\usepackage{graphicx}
\usepackage{textcomp}
\usepackage{float}

\newtheorem{theorem}{Theorem}
\newtheorem{lemma}{Lemma}
\newtheorem{definition}{Definition}
\newtheorem{remark}{Remark}

\begin{document}
\title{Finite horizon stochastic $H_2/H_\infty$ control for continuous-time mean-field systems with Poisson jumps}
\author{Huimin Han,
 Shaolin Ji,
 and Weihai Zhang, \IEEEmembership{Senior Member, IEEE}
\thanks{Corresponding authors: Shaolin Ji;Weihai Zhang.}
\thanks{Huimin Han and Shaolin Ji are with Zhongtai Securities Institute for Financial Studies, Shandong University, Jinan 250100 P. R. China (e-mails: hanhuiminhhm@mail.sdu.edu.cn, jsl@sdu.edu.cn). }
\thanks{Weihai Zhang is with College of Electrical Engineering and Automation, Shandong University of Science and Technology, Qingdao 266590 P. R. China (e-mail: w\_hzhang@163.com).}}
\maketitle
\begin{abstract}
The stochastic $H_2/H_\infty$ control problem for continuous-time mean-field stochastic differential equations with Poisson jumps over finite horizon is investigated in this paper. Continuous and jump diffusion terms in the system depend not only on the state but also on the control input, external disturbance, and mean-field components. By employing the quasi-linear technique and the method
of completing the square, a mean-field stochastic jump bounded real lemma of the system is derived, which plays a crucial role in solving stochastic $H_2/H_\infty$ control problem. It is demonstrated in this study
that the feasibility of the stochastic $H_2/H_\infty$ control problem is equivalent to the solvability of four sets of cross-coupled generalized differential Riccati equations, thus generalizing the previous results to mean-field jump-diffusion systems. To validate the proposed
methodology, a numerical simulation example is provided to illustrate the effectiveness of the control strategy. The results establish a systematic approach for designing $H_2/H_\infty$ controllers that simultaneously guarantee the robustness against disturbances and optimal performance for interacting particle systems.
\end{abstract}
\section{Introduction}
\label{sec:introduction}
This paper addresses the design of stochastic $H_{2}/H_{\infty}$ control
for mean-field Poisson jump systems over a finite time horizon, as well as the model-free $H_{\infty}$ control design for linear mean-field systems.
Consider System (\ref{xstateequation}), where $u(t)$, $v(t)$, and $z(t)$ denote the control input, external disturbance, and controlled output, respectively. $\{\tilde N_{p}\}_{0\leq t \leq T}$ denote a Poisson random martingale measure and $\{W(t)\}_{0\leq t \leq T}$ represent a one-dimensional standard Brownian motion defined on the complete probability space $(\Omega,\mathcal{F},P)$. The system dynamics are then given by:
\begin{equation}\label{xstateequation}
\left\{
\begin{aligned}
&dx(t)=\left\{A(t)x(t)+\bar{A}(t)\mathbb{E}[x(t)]+B_{2}(t)u(t)+\right.\\
&\left.\bar{B}_{2}(t)\mathbb{E}[u(t)]+
B_{1}(t)v(t)+\bar{B}_{1}(t)\mathbb{E}[v(t)]\right\}dt +\left\{C(t)x(t)\right.\\
&+\bar{C}(t)\mathbb{E}[x(t)]+D_{2}(t)u(t)+\bar{D}_{2}(t)\mathbb{E}[u(t)]+
D_{1}(t)v(t)+\\
&\left.\bar{D}_{1}(t)\mathbb{E}[v(t)]\right\}dW(t)+\int_G\left\{E(t,\theta)x(t-)+\right.\\
&\bar{E}(t,\theta)\mathbb{E}[x(t-)]+F_{2}(t,\theta)u(t)+\bar{F_{2}}(t,\theta)\mathbb{E}[u(t)]+\\
&\left.F_{1}(t,\theta)v(t)+\bar{F_{1}}(t,\theta)\mathbb{E}[v(t)]\right\}\tilde{N}_p(d\theta,dt),\\
&x(0)=x_{0}, \\
&z(t)=\left(
\begin{array}{l}
M(t) x(t) \\
N(t) u(t)
\end{array}\right),
\end{aligned}
\right.
\end{equation}
where $T< \infty, t\in[0,T],$ and $N^{\prime}(t) N(t)=I$. All coefficients in (\ref{xstateequation}) constitute deterministic continuous matrices with appropriate dimensions.

The distinctive feature of equation (\ref{xstateequation}) lies in its incorporation of mean-field terms $\mathbb{E}[x(t)]$, $\mathbb{E}[u(t)]$, and $\mathbb{E}[v(t)]$, which fundamentally differentiates it from conventional stochastic differential equations (SDEs). Research on this system and its applications includes mean-field games \cite{2007MF}, \cite{mfg}, mean-field backward stochastic differential equations \cite{2009},\cite{2009PDE} and mean-field control systems \cite{2011},\cite{2011mf},\cite{2013},\cite{2013a} among others. Poisson jump processes model discontinuous dynamics prevalent in practical applications, particularly in financial markets. System (\ref{xstateequation}) is particularly
important in the field of financial optimization, as it accurately describes risky asset pricing and strategic interactive behaviors. We refer the reader to \cite{jump} to obtain comprehensive information about financial modeling with jump. Although existing literature extensively addresses control problems \cite{2015},  \cite{wang_2023}, \cite{stackelberg}, there has been limited research on robust control for (\ref{xstateequation}). Based on the assumptions that disturbances cannot be detected and optimization is performed under worst-case disturbance scenarios, this paper investigates the stochastic $H_2/H_\infty$ control problem for system (\ref{xstateequation}) using a Nash game approach.

Our results extend the framework for SDEs by incorporating mean-field terms $\mathbb{E}[x(t)],\mathbb{E}[u(t)]$, and $\mathbb{E}[v(t)]$ into the state equation, leading to coupled Riccati equations (\ref{RE1})-(\ref{RE22}) with respect to $(P,Q)$ instead of the Riccati equations solely dependent on $P$. Within the linear stochastic differential game framework, introducing mean-field terms while having diffusion terms involving $(u, v)$ imposes stringent conditions where multiple algebraic equations ($\Sigma_0(P_1), \Sigma_2(P_1), \tilde{\Sigma}_0(P_2), \tilde{\Sigma}_2(P_2))$ in (\ref{RE1})-(\ref{RE22}) must simultaneously satisfy positive definiteness—a fundamental difficulty highlighted in the literature \cite{2015}. The complexity arising from the predictable property induced by Poisson jump processes is also considered and resolved. The
disturbances $v$ affect the output through the system states, which ensures the attenuation of the high-frequency gain in the frequency domain. As a result, the $H_\infty$-norm of the transfer function has no lower bound, and this guarantees the validity of the necessary and sufficient conditions in our conclusion, where the disturbance level $\gamma$ can be set by us. We also propose a numerical method based on backward-in-time integration and provide example verification along with visualization results.

The rest of the paper is organized as follows. Section 2 formulates the problem and introduces key notations, definitions, and preliminary lemmas. Section 3 develops the Stochastic Bounded Real Lemma for mean-field jump-diffusion systems (MF-SJBRL). Section 4 describes the main results. Using classical linear quadratic control results and MF-SJBRL, the feedback representations are obtained through solving four coupled sets of generalized differential Riccati equations (GDREs). Numerical methods are presented in Section 5, followed by the final conclusion in Section 6.

\section{Preliminaries}
\subsection{Notations}
\begin{enumerate}
  \item $\mathbb{R}^{n}$ is the real $n$-dimensional space and $\mathbb{R}^{n \times m}$ is the space of all real $n \times m$ matrices. $I$ is the identity matrix in $\mathbb{R}^{n \times n}$. For a matrix/vector $A, A^{\prime}$ denotes the transpose of $A$, and $|A|$ denotes the square root of the summarized squares of all the components of the matrix/vector A. For square matrix $A, \operatorname{det}(A)$ is the determinant of $A$, and $A^{-1}$ is its inverse if $A$ is nonsingular. $A>0 (A\geqslant 0) /A<0$ means that $A$ is a positive definite (semi-definite) /negative definite symmetric matrix. $\left\langle A_{1}, A_{2}\right\rangle$ denotes the inner product of two vectors $A_{1}$ and $A_{2}$.
  \item $(\Omega, \mathcal{F}, \{\mathcal{F}_{t}\}_{t\geq 0}, P)$ is a given complete filtered probability space, on which a one-dimensional standard Brownian motion $\{W(t)\}_{0\leq t\leq T}$ and a compensated Poisson random measure $\tilde{N}_{p}$ are defined and assumed to be mutually independent.
    $$
    \begin{aligned}
    \mathcal{F}_{t}:=& \sigma[W(s) ; 0 \leqslant s \leqslant t]\vee \\
    &\sigma\left[\iint_{A \times(0, s]} \tilde{{N}_{p}}(d \theta, dr) ; 0 \leqslant s \leqslant t\right], A \in \mathcal{B}(G)
    \end{aligned}
    $$
    is $P$-completed filtration. More specifically, denote by $\mathcal B(\Lambda)$ the Borel $\sigma$-algebra of any topological space $\Lambda.$ Let $(G,\mathcal{B}(G), \nu)$ be a measurable space, $\nu$ is a measure with $\nu(A)<\infty$ for any $A \in \mathcal{B}(G).$ And $p: \Omega\times D_p \longrightarrow G$ is a $\mathcal{F}_t$-adapted stationary Poisson point process with characteristic measure $\nu$, where $D_p$  is a countable subset of $(0, \infty)$. Then the counting measure induced by $p$ is
    $$
    N_p((0,t]\times A):=\#\{s\in D_p; s\leq t, p(s)\in A\},
    $$
    for $t>0, A\in \mathcal{ B }(G)$.
    Let
    $
    \tilde{N}_{p}(d\theta,dt):=N_p(d\theta,dt)-\nu(d\theta)dt
    $ be a compensated Poisson random martingale  measure.
\item $S^n \in \mathbb{R}^{n \times n}$: the collection of $n\times n$ real symmetric matrices.

$S^n_{+} \in \mathbb{R}^{n \times n}$: the set of all non-negative definite matrices of $S^n$.

$L^{\infty}\left([0, T], \mathbb{R}^{n\times n}\right)$: the collection of $\mathbb{R}^{n\times n}$-valued, processes $\eta(t)$ with $||\eta(t)||_\infty:=\operatorname{ess\,sup}_{t\in [0,T]}|\eta(t)|<+\infty.$

$L_{\mathcal{F}_{t}}^2(\Omega;\mathbb{R}^d)$: the collection of $\mathbb{R}^d$-valued, $\mathcal{F}_{t}$-measurable random variables $\eta$
with $$||\eta||^{2}:=\mathbb{E}[|\eta|^2]<+\infty.$$

$L_{\mathcal{F}}^2([0,T];\mathbb{R}^d)$: the collection of $\mathbb{R}^d$-valued, $\mathcal{F}$-adapted
 random processes $\eta(t)$ with $$||\eta(t)||_{[0,T]}^2:=\mathbb{E}\int_{0}^{T}|\eta(t)|^2dt<+\infty.$$

$\mathcal{H}^2_{\mathcal{F}}([0,T];\mathbb{R}^d)$: the space of all $\mathbb{R}^d$-valued, $\mathcal{F}$-adapted $c\grave{a}dl\grave{a}g$ process $\varphi(t)$ on $[0,T]$, such that  $$\mathbb{E}\int_0^T|\varphi(t)|^2dt<+\infty.$$

$\mathcal{S}^2_{\mathcal{F}}([0,T];\mathbb{R}^d)$: the space of all $\mathbb{R}^d$-valued, $\mathcal{F}$-adapted $c\grave{a}dl\grave{a}g$
random process $\varphi(t)$
on $[0,T]$ with $$\mathbb{E}[\sup_{0\leq s\leq T}|\varphi(t)|^2]<+\infty.$$

$M^{\nu,2}\left([0, T]\times G, \mathbb{R}^{n\times n}\right)$: the collection of $\mathbb{R}^{n\times n}$-valued, $r(t,\theta)$ with
$$||r(t,\theta)||_{M^{\nu,2}}^2:=\int_{0}^{T}\int_{G}\|r(t,\theta)\|^2\nu(d\theta)dt<+\infty.$$

$M_{\mathcal{F}}^{\nu,2}\left([0, T]\times G, \mathbb{R}^{n\times n}\right)$: the collection of $\mathbb{R}^{n\times n}$-valued, $\mathcal{F}$-predictable random processes $r(t,\omega,\theta)$ with
$$||r(t,\theta)||_{M_{\mathcal{F}}^{\nu,2}}^2:=\mathbb{E}\int_{0}^{T} \int_{G}\|r(t,\theta)\|^2\nu(d\theta)dt<+\infty.$$

$\mathcal{P}_{n}(\mathbb{R}^{n})$: the collection of probability measures $\nu$ on $\mathbb{R}^{n}$ with finite second order moment, i.e. $$\int|x|^n\nu(dx)<+\infty.$$

$\mathcal{U}([0, T];\mathbb{R}^n)$: the collection of predictable processes $u(t)$, which belongs to $L_{\mathcal{F}}^2([0,T];\mathbb{R}^n)$.
The above notations are adopted in the subsequent analysis.
\end{enumerate}
\subsection{Problem Formulation}
Assuming that $A, \bar{A}, C, \bar{C},A_{11}, \bar{A}_{11},C_{11}, \bar{C}_{11}\in$
$L^{\infty}\left([0, T], \mathbb{R}^{n\times n}\right)$, $B_{2}, \bar{B}_{2}, D_{2}, \bar{D}_{2}\in L^{\infty}\left([0, T], \mathbb{R}^{n\times n_{u}}\right),$
$B_{1}, \bar{B}_{1}, D_{1}, \bar{D}_{1},B_{11}, \bar{B}_{11},D_{11}, \bar{D}_{11} \in$ $L^{\infty}\left([0, T], \mathbb{R}^{n\times n_{v}}\right),$
$E,\bar{E},E_{11},\bar{E}_{11} \in M^{\nu,2}\left([0, T]\times G,\mathbb{R}^{n\times n}\right),$
$F_{2},\bar{F}_{2}$$\in$ $M^{\nu,2}\left([0, T]\times G, \mathbb{R}^{n\times n_{u}}\right),$
$ F_{1},\bar{F}_{1}, F_{11},\bar{F}_{11}\in$ $M^{\nu,2}\left([0, T]\times G, \mathbb{R}^{n\times n_{v}}\right).$ Above standard assumptions will be in force throughout this paper.

\begin{definition}\label{def}(see \cite{2004}) For $0<T<\infty$, by Lemma \ref{lemmaxstatesolution}, when $\left(u, v, x_{0}\right) \in \mathcal{U}([0, T];\mathbb{R}^{n_{u}}) \times \mathcal{U}([0, T];\mathbb{R}^{n_{v}}) \times$ $\mathbb{R}^{n}$, there exists a unique solution $x(t):x\left(t, u, v, x_{0}\right) \in$ $\mathcal{H}_{\mathcal{F}}^{2}\left([0, T]; \mathbb{R}^{n}\right)$. The finite horizon stochastic $H_{2} / H_{\infty}$ control problem  of (\ref{xstateequation}) can be stated as follows.

Given disturbance attenuation $\gamma>0$, to find $u^{*}(t, x) \in \mathcal{U}\left([0,T]; \mathbb{R}^{n_{u}}\right)$, such that

(1)
\begin{equation}\label{hinf}
\begin{aligned}
\|\mathcal{L}_{u^{*}}\|_{[0,T]}&=\sup _{\substack{v \in \mathcal{U}\left([0,T]; \mathbb{R}^{n_{v}}\right) \\
v \neq 0, x_{0}=0}} \frac{\|z\|_{[0,T]}}{\|v\|_{[0,T]}} \\
&:=\sup _{\substack{v \in \mathcal{U}\left([0,T]; \mathbb{R}^{n_{v}}\right) \\
v \neq 0, x_{0}=0}} \frac{\left\{\mathbb{E}\int_{0}^{T}\left(x^{\prime} M^{\prime} M x+u^{* \prime} u^{*}\right) d t\right\}^{1 / 2}}{\left\{\mathbb{E} \int_{0}^{T} v^{\prime} v d t\right\}^{1 / 2}} \\
&<\gamma,
\end{aligned}
\end{equation}
where
\begin{equation*}
\mathcal{L}_{u^{*}}(v)=\left[\begin{array}{c}
M x\left(t, u^{*}, v, 0\right) \\
N u^{*}
\end{array}\right]
\end{equation*}
is called the perturbation operator of (\ref{xstateequation}).

(2) When the worst-case disturbance $v^{*}(t, x) \in$ $\mathcal{U}([0, T];\mathbb{R}^{n_{v}})$, if it exists, is applied to system (\ref{xstateequation}), $u^{*}(t, x)$ minimizes the output energy
      \begin{equation}\label{j2}
      J_{2}\left(u, v^{*};0,x_{0}\right)=\|z\|_{2}^{2}=\mathbb{E} \int_{0}^{T}\left(x^{\prime} M^{\prime} M x+u^{\prime} u\right) dt.\end{equation}
Here, $v^{*}(t, x)$ is called a worst-case disturbance in the sense that
    \begin{equation}\label{j1}
    \begin{aligned}
    &v^{*}(t, x) =\mathop{\arg\min}\limits_{v \in \mathcal{U}\left([0,T]; \mathbb{R}^{n_{v}}\right)} J_{1}\left(u^{*}, v;0,x_{0}\right) ,\forall x_{0} \in L_{\mathcal{F}_{0}}^{2}\left(\Omega; \mathbb{R}^{n}\right), \\
    &J_{1}\left(u^{*}, v;0,x_{0}\right)=\mathbb{E} \int_{0}^{T}\left(\gamma^{2} v^{\prime} v-z^{\prime} z\right) d t .
    \end{aligned}
    \end{equation}
If the previous predictable progresses $\left(u^{*}, v^{*}\right)$ exists, then we say that the finite horizon $H_{2} / H_{\infty}$ control admits a pair of solutions.

\end{definition}
\begin{remark}(see \cite{2004})
If the finite horizon stochastic $H_{2} / H_{\infty}$ control problem is solvable, then the global Nash equilibrium strategies $\left(u^{*}, v^{*}\right)$ for two-player, nonzero-sum game satisfying
\begin{equation}\label{nash1}
J_{1}\left(u^{*}, v^{*}\right) \leq J_{1}\left(u^{*}, v\right)
\end{equation}
and
\begin{equation}\label{nash2}
J_{2}\left(u^{*}, v^{*}\right) \leq J_{2}\left(u, v^{*}\right).
\end{equation}
 To guarantee the uniqueness of the global Nash equilibrium in (\ref{nash1}) and (\ref{nash2}), both players are restricted to using linear, memoryless state feedback control.
\end{remark}
\subsection{Three Useful Lemmas}
Now we recall three lemmas for SDE driven by the Brownian motion and Poisson random jump.
\begin{lemma}\label{lemmaxstatesolution}(see \cite{1992} and \cite{2019})
Assume that $b:[0, T] \times \mathbb{R}^{n} \times \mathcal{P}_{2}(\mathbb{R}^{n})\rightarrow \mathbb{R}^{n}, \sigma:[0, T] \times \mathbb{R}^{n} \times \mathcal{P}_{2}(\mathbb{R}^{n}) \rightarrow \mathbb{R}^{n}$, $\pi:[0, T] \times G \times \mathbb{R}^{n} \times \mathcal{P}_{2}(\mathbb{R}^{n}) \rightarrow \mathbb{R}^{n}$ satisfy the following conditions:

  (1)
  $b(\cdot,0,0) \in L_{\mathcal{F}}^{2}\left([0, T] ; \mathbb{R}^{n}\right)$, $\sigma(\cdot,0,0) \in L_{\mathcal{F}}^{2}\left([0, T] ; \mathbb{R}^{n}\right)$, and $\pi(\cdot,\cdot,0,0) \in M_{\mathcal{F}}^{\nu, 2}\left([0, T] \times G ; \mathbb{R}^{n}\right)$;

  (2)
  For all $x, \bar{x} \in \mathbb{R}^{n}$, $\mu_{1}, \mu_{2} \in \mathcal{P}_{2}(\mathbb{R}^{n})$, $t \in [0,T]$, there exists a constant $C_T>0$ such that
\begin{equation*}
\begin{aligned}
&|b(t, x, \mu_{1})-b(t, \bar{x}, \mu_{2})|+|\sigma(t, x, \mu_{1})-\sigma(t, \bar{x}, \mu_{2})|\\
& +\left(\int_{G}|\pi(t, \theta, x, \mu_{1})-\pi(t, \theta, \bar{x}, \mu_{2})|^{2} \nu(d \theta)\right)^{\frac{1}{2}} \\
\leqslant &C_T(|x-\bar{x}|+\rho(\mu_{1},\mu_{2})),
\end{aligned}
\end{equation*}
where $\rho(\mu_{1},\mu_{2})$ is Wasserstein metric which satisfies
\begin{equation*}
\begin{aligned}
&\rho(\mu_{1},\mu_{2})\\
= &\inf \left\{\int |x-y|r(dx, dy); r \; has \; marginals \; \mu_{1} \; and \; \mu_{2}\right\}\\
= &\sup\left\{\langle g, \mu_{1}\rangle -\langle g, \mu_{2}\rangle; g(x) - g(y)\leq |x-y|\right\}.
\end{aligned}
\end{equation*}
Then the following equation with the Poisson random jumps
\begin{equation*}
\begin{aligned}
&x(t)=x_{0}+\int_{0}^{t} b\left(s, x(s), \mu(s)\right) d s\\
&+\int_{0}^{t} \sigma\left(s, x(s), \mu(s)\right) d W(s)\\
&+\int_{0}^{t}\int_{G} \pi\left(s, \theta, x(s-), \mu(s)\right) \tilde{N_{p}}(d \theta, d s)
\end{aligned}
\end{equation*}
admits a unique strong solution $x \in \mathcal{S}_{\mathcal{F}}^{2}\left(0, T ; \mathbb{R}^{n}\right)$.

Moreover, for system (\ref{SBRLxstate}), the following estimate holds:
\begin{equation}\label{estimate}
\begin{aligned}
\mathbb{E}\left\{\sup _{0 \leqslant t \leqslant T}|x(t)|^{2}\right\} \leqslant K \mathbb{E}\left\{|x(0)|^{2}+\int_{0}^{T}\left|v(t)\right|^{2} d t \right\},
\end{aligned}
\end{equation}
where $K>0$ is a constant relying on the Lipschitz constant $C_T$ and the time horizon $T$.
\end{lemma}
\begin{lemma}\label{lemmasemiito}
 (Generalized \mbox{It\^o} formula)
 Let $x(t)$ satisfy
\begin{equation*}
\begin{aligned}
d x(t)=&b\left(t, x(t)\right) d t+\sigma\left(t, x(t)\right) d W(t)\\
&+\int_{G} c\left(t, \theta, x(t-)\right) \tilde{N_{p}}(d \theta, d t),
\end{aligned}
\end{equation*}
and $\phi(\cdot, \cdot) \in C^{1,2}\left([0, T] \times \mathbb{R}^{n}\right)$. Then
\begin{equation}\label{semiito}
\begin{aligned}
&d\phi\left(t, x(t)\right)=\phi_{t}\left(t, x\right) d t+\left\langle\phi_{x}\left(t, x\right), b\left(t, x\right)\right\rangle dt+\\
&\left\langle\phi_{x}\left(t, x\right), \sigma\left(t, x\right)\right\rangle d W(t)+\frac{1}{2} \operatorname{tr}\left\{\sigma^{\prime}(t, x) \phi_{x x}(t, x) \sigma(t, x)\right\} dt\\
&+\int_{G}\left[\phi(t, x+c(t, \theta, x))-\phi(t, x)-\right.\\
&\left.\left\langle\phi_{x}(t, x), c(t, \theta, x)\right\rangle\right] \nu(d \theta)dt\\
&+\int_{G}\left[\phi\left(t, x(t-)+c\left(t, \theta, x(t-)\right)\right)-\phi\left(t, x(t-)\right)\right] \tilde{N_{p}}(d \theta, d t),
\end{aligned}
\end{equation}
and $\phi_{t}$ and $\phi_{x}$ denote the partial derivatives of $\phi$ with respect to $t$ and $x$ respectively, and $\phi_{x x}$ denotes the second-order partial derivative of $\phi$ with respect to $x$.
\end{lemma}
\begin{lemma}\label{lemmaPpositive}
Consider the following differential equation:T
\begin{equation*}
\left\{
\begin{aligned}
&\dot{P}+P \tilde{A}+\tilde{A}^{\prime} P+\tilde{C}^{\prime} P \tilde{C}+\tilde{Q}\\
&+\int_{G}\left(\tilde{E}^{\prime}(\theta)P \tilde{E}(\theta)\right)\nu(d\theta)=0 ,\\
&P(T)=\tilde{G}, \quad t \in[0, T],
\end{aligned}
\right.
\end{equation*}
where $\tilde{A}, \tilde{C} \in L^{\infty}\left(0, T ; \mathbb{R}^{n \times n}\right)$, $\tilde{E} \in M^{\nu,2}\left([0, T]\times G, \mathbb{R}^{n\times n}\right)$, $\tilde{G} \in S_{+}^{n}, \tilde{Q} \in L^{\infty}\left(0, T ; S_{+}^{n}\right)$. Then, the equation admits a unique solution $P \in C\left([0, T] ; S_{+}^{n}\right)$.
\end{lemma}
\section{Mean-Field Stochastic Jump Bounded Real Lemma}
A finite horizon mean-field stochastic jump bounded real lemma (for short, MF-SJBRL) is obtained in this section, which serves as a crucial tool for analyzing stochastic $H_{2} / H_{\infty}$ control. For stochastic system
\begin{equation}\label{SBRLxstate}
\left\{
\begin{aligned}
&dx(t)=\left\{A_{11}(t)x(t)+\bar{A}_{11}(t)\mathbb{E}[x(t)]+B_{11}(t)v(t)+\right.\\
&\left.\bar{B}_{11}(t)\mathbb{E}[v(t)]\right\}dt +\left\{C_{11}(t)x(t)+\bar{C}_{11}(t)\mathbb{E}[x(t)]+\right.\\
&\left.D_{11}(t)v(t)+\bar{D}_{11}(t)\mathbb{E}[v(t)]\right\}dW(t)+\int_G\left\{E_{11}(t,\theta)x(t-)\right.\\
&+\bar{E}_{11}(t,\theta)\mathbb{E}[x(t-)]+F_{11}(t,\theta)v(t)+\\
&\left.\bar{F_{11}}(t,\theta)\mathbb{E}[v(t)]\right\}\tilde{N}_p(d\theta,dt),\\
&x(0)=x_{0} \in\mathbb{R}^{n},\\
&z_{1}(t)= M_{11} x(t), \quad  t\in[0,T], \\
\end{aligned}
\right.
\end{equation}
define the perturbation operator as
\begin{equation}
\begin{aligned}
\|\tilde{\mathcal{L}}\|_{[0,T]}&=\sup _{\substack{v \in \mathcal{U}([0, T];\mathbb{R}^{n_{v}}) \\
v \neq 0, x_{0}=0}} \frac{\|z_{1}\|_{[0,T]}}{\|v\|_{[0,T]}} \\
&:=\sup _{\substack{v \in \mathcal{U}([0, T];\mathbb{R}^{n_{v}}) \\
v \neq 0, x_{0}=0}} \frac{\left\{\mathbb{E} \int_{0}^{T}\left(x^{\prime} M_{11}^{\prime} M_{11} x\right) d t\right\}^{1 / 2}}{\left\{\mathbb{E} \int_{0}^{T} v^{\prime} v d t\right\}^{1 / 2}},
\end{aligned}
\end{equation}
and cost funtional
\begin{equation}
J_{1}\left(0, v;0,x_{0}\right)=\mathbb{E} \int_{0}^{T}\left(\gamma^{2} v^{\prime} v-z_{1}^{\prime} z_{1}\right) d t .
\end{equation}
\begin{lemma}{(MF-SJBRL)}\label{SBJRL}
$\|\tilde{\mathcal{L}}\|_{[0, T]}<\gamma$ for some $\gamma>0$ iff the following differential Riccati equations (DRE) (with the time argument $t$ suppressed)
\begin{equation}\label{3.4}
\begin{aligned}
&\begin{cases}
\mathcal{S}(P)-\mathcal{G}(P)\Sigma_{0}^{-1}(P) \mathcal{G}^{\prime}(P)=0,\\
P(T)=0,\\
\Sigma_{0}(P)>0,
\end{cases}
\end{aligned}
\end{equation}
\begin{equation}\label{3.5}
\begin{aligned}
&\begin{cases}
\tilde{\mathcal{S}}(P,Q) -\tilde{\mathcal{G}}(P,Q)\Sigma_{2}^{-1}(P)\tilde{\mathcal{G}}^{\prime}(P,Q)=0,\\
Q(T)=0,\\
\Sigma_{2}(P)>0.
\end{cases}
\end{aligned}
\end{equation}
have unique solution $P,Q\leq 0$ on $[0, T]$, where
\begin{equation}
\begin{aligned}
&\mathcal{S}(P)=\dot{P}+P A_{11}+A_{11}^{\prime} P+C_{11}^{\prime} P C_{11} \\
&+\int_G\{E_{11}(\theta)^{\prime}P E_{11}(\theta)\}\nu(d\theta) -M_{11}^{\prime} M_{11},\\
&\mathcal{G}(P)=P B_{11}+C_{11}^{\prime} P D_{11} +\int_G\{E_{11}(\theta)^{\prime}P F_{11}(\theta)\}\nu(d\theta),\\
&\Sigma_{0}(P)=\gamma^{2}I+D_{11}^{\prime} P D_{11}+\int_G\{F_{11}(\theta)^{\prime}P F_{11}(\theta)\}\nu(d\theta),\\
&\tilde{\mathcal{S}}(P,Q)=\dot{Q}+Q(A_{11}+\bar{A_{11}})+(A_{11}+\bar{A_{11}})^{\prime} Q\\
&+(C_{11}+\bar{C_{11}})^{\prime} P(C_{11}+\bar{C_{11}})\\
&+\int_G\{(E_{11}+\bar{E_{11}})(\theta)^{\prime}P (E_{11}+\bar{E_{11}})(\theta)\}\nu(d\theta)
-M_{11}^{\prime} M_{11},\\
&\tilde{\mathcal{G}}(P,Q)=Q (B_{11}+\bar{B_{11}})+(C_{11}+\bar{C_{11}})^{\prime} P (D_{11}+\bar{D_{11}}) \\
& +\int_G\{(E_{11}+\bar{E_{11}})(\theta)^{\prime}P (F_{11}+\bar{F_{11}})(\theta)\}\nu(d\theta),\\
&\Sigma_{2}(P)=\gamma^{2}I+(D_{11}+\bar{D_{11}})^{\prime} P(D_{11}+\bar{D_{11}})\\
&+\int_G\{(F_{11}+\bar{F_{11}})(\theta)^{\prime}P (F_{11}+\bar{F_{11}})(\theta)\}\nu(d\theta).
\end{aligned}
\end{equation}
\end{lemma}
For convenience, we denote $\Phi=-\Sigma_{0}(P)^{-1}\mathcal{G}^{\prime}(P), \Psi=-\Sigma_{2}(P)^{-1}\tilde{\mathcal{G}}^{\prime}(P,Q).$
Before proving the MF-SJBRL, we propose some lemmas.
\begin{lemma}\label{lemmalinearj1}
Assume that $\varphi,\psi \in C\left([0, T] ; \mathbb{R}^{n_v \times n}\right)$ and $P^{\gamma, \varphi} ,Q^{\gamma, \varphi, \psi}\in C\left([0, T] ; S^{n}\right)$ satisfy the following linear differential matrix-valued equations
\begin{equation}\label{SBRLlinearRE1}
\begin{cases}
                                \left(\begin{array}{c}
                                   I \\
                                   \varphi
                                 \end{array}\right)^{\prime}\left(
\begin{array}{cc}
  \mathcal{S}(P^{\gamma, \varphi})  & \mathcal{G}(P^{\gamma, \varphi}) \\
  \mathcal{G}^{\prime}(P^{\gamma, \varphi})& \Sigma_{0}(P^{\gamma, \varphi})
\end{array}
\right)
                                 \left(
                                 \begin{array}{c}
                                   I \\
                                   \varphi
                                 \end{array}
                                 \right)=0,\\
                                 P^{\gamma, \varphi}(T)=0,
                                 \end{cases}
\end{equation}
\begin{equation}\label{SBRLlinearRE2}
\begin{cases}
                                 \left(\begin{array}{c}
                                   I \\
                                   \psi
                                 \end{array}\right)^{\prime}\left(
\begin{array}{cc}
  \tilde{\mathcal{S}}(P^{\gamma, \varphi},Q^{\gamma, \varphi, \psi})  & \tilde{\mathcal{G}}(P^{\gamma, \varphi},Q^{\gamma, \varphi, \psi}) \\
  \tilde{\mathcal{G}}^{\prime}(P^{\gamma, \varphi},Q^{\gamma, \varphi, \psi})& \Sigma_{2}(P^{\gamma, \varphi})
\end{array}
\right)
                                 \left(\begin{array}{c}
                                   I \\
                                   \psi
                                 \end{array}\right)\\
                                 =0,\\
                                 Q^{\gamma, \varphi, \psi}(T)=0.
                                 \end{cases}
\end{equation}
Then, for any $(\tau, \xi)\in[0,T]\times L^{2}_{\mathcal{F}_{\tau}}(\Omega; \mathbb{R}^{n})$, $v \in \mathcal{U}([\tau, T];\mathbb{R}^{n_{v}})$, we derive that
\begin{equation}\label{SBRLJ1linear}
\begin{aligned}
& J_{1}\left(0,v+\varphi(x^{\varphi,\psi}-\mathbb{E}x^{\varphi,\psi})+ \psi\mathbb{E}x^{\varphi,\psi}; \tau, \xi\right) \\
=&\mathbb{E}\left\langle(\xi-\mathbb{E}\xi), P_{\tau}^{\gamma, \varphi}(\xi-\mathbb{E}\xi)\right\rangle+\left\langle\mathbb{E}\xi, Q_{\tau}^{\gamma,\varphi,\psi}\mathbb{E}\xi\right\rangle\\
&+\mathbb{E}\int_{\tau}^{T}\left\langle(v(t)-\mathbb{E}v(t)),(\mathcal{G}^{\prime}(P^{\gamma, \varphi})+\Sigma_{0}(P^{\gamma, \varphi})\varphi(t))\cdot\right.\\
&\left.(x^{\varphi,\psi}(t)-\mathbb{E}x^{\varphi,\psi}(t))\right\rangle+\left\langle(\mathcal{G} ^{\prime}(P^{\gamma, \varphi})+\Sigma_{0}(P^{\gamma, \varphi})\varphi(t))\cdot\right.\\
&\left.(x^{\varphi,\psi}(t)-\mathbb{E}x^{\varphi,\psi}(t)),v(t)-\mathbb{E}v(t)\right\rangle\\
&+\left\langle v(t)-\mathbb{E}v(t), \Sigma_{0}(P^{\gamma, \varphi})(v(t)-\mathbb{E}v(t)) \right\rangle \\
&+\left\langle\mathbb{E}v(t),(\tilde{\mathcal{G}}^{\prime}(P^{\gamma, \varphi},Q^{\gamma,\varphi,\psi})+\Sigma_{2}(P^{\gamma, \varphi})\psi(t))\mathbb{E}x^{\varphi,\psi}(t)\right\rangle\\
&+\left\langle(\tilde{\mathcal{G}}^{\prime}(P^{\gamma, \varphi},Q^{\gamma,\varphi,\psi})+\Sigma_{2}(P^{\gamma, \varphi})\psi(t))\mathbb{E}x^{\varphi,\psi}(t),\mathbb{E}v(t)\right\rangle\\
&+\left\langle\mathbb{E}v(t), \Sigma_{2}(P^{\gamma, \varphi})\mathbb{E}v(t)\right\rangle dt,
\end{aligned}
\end{equation}
where $$x^{\varphi,\psi}(t, v(\cdot) ; \tau, \xi)=x\left(t, v+\varphi(x^{\varphi,\psi}-\mathbb{E}x^{\varphi,\psi})+ \psi\mathbb{E}x^{\varphi,\psi} ; \tau, \xi\right)$$ solves (\ref{SBRLxstate}). In particular,
\begin{equation}
\begin{aligned}
&J_{1}\left(0,\varphi(x^{\varphi,\psi}-\mathbb{E}x^{\varphi,\psi})+ \psi\mathbb{E}x^{\varphi,\psi}; \tau, \xi\right)\\
=&\mathbb{E}\left\langle(\xi-\mathbb{E}\xi), P_{\tau}^{\gamma,\varphi}(\xi-\mathbb{E}\xi)\right\rangle+\left\langle\mathbb{E}\xi, Q_{\tau}^{\gamma,\varphi}\mathbb{E}\xi\right\rangle.
\end{aligned}
\end{equation}
\end{lemma}

\begin{lemma}\label{lemmainitial}
If $\|\tilde{\mathcal{L}}\|<\gamma$, then for any $(\tau, \xi)\in[0,T]\times L^{2}_{\mathcal{F}_{\tau}}(\Omega; \mathbb{R}^{n}), v\in\mathcal{U}([\tau,T];\mathbb{R}^{n_{v}})$, there exists $\mu> 0$ such that $J_{1}\left(0, v; \tau, \xi\right)\geq -\mu \mathbb{E}|\xi|^{2} $.
\end{lemma}

\begin{lemma}\label{main}
If $\|\tilde{\mathcal{L}}\|<\gamma$, $\varphi,\psi \in C\left([0, T] ; \mathbb{R}^{n_v \times n}\right)$ and $P^{\gamma, \varphi} ,Q^{\gamma, \varphi,\psi}\in C\left([0, T ]; S^{n}\right)$ are the solutions of linear differential matrix-valued equations (\ref{SBRLlinearRE1})-(\ref{SBRLlinearRE2}). Then for any $\delta > 0$ satisfying $\delta <\gamma^{2}- \|\tilde{\mathcal{L}}\|^{2}$,  $\Sigma_{0}(P^{\gamma, \varphi})\geq \delta I$, $\Sigma_{2}(P^{\gamma, \varphi})\geq \delta I.$
\end{lemma}

\section{Stochastic $H_{2}/H_{\infty}$ control }
\begin{theorem}\label{mainresult}
Finite horizon $H_{2} / H_{\infty}$ control has solution $\left(u^{*}(t, x), v^{*}(t, x)\right)$, where $u^{*}(t, x)$ and $v^{*}(t, x)$ are the following time-variant feedback strategies:
\begin{equation}
\begin{aligned}
&u^{*}(t, x)=K_{2}(t)x(t-)+\tilde{K_{2}}(t)\mathbb{E}(x(t-))\\
&=K_{2}(t)[x(t-)-\mathbb{E}(x(t-))]+(K_{2}(t)+\tilde{K_{2}}(t))\mathbb{E}(x(t-)), \\
&v^{*}(t, x)=K_{1}(t)x(t-)+\tilde{K_{1}}(t)\mathbb{E}(x(t-))\\
&=K_{1}(t)[x(t-)-\mathbb{E}(x(t-))]+(K_{1}(t)+\tilde{K_{1}}(t))\mathbb{E}(x(t-)),
\end{aligned}
\end{equation}
respectively, iff the four sets of coupled Riccati equations
\begin{equation}\label{RE1}
\left\{
\begin{array}{l}
\mathcal{S}_{1}(P_{1})-\mathcal{G}_{1}(P_{1})\Sigma_{0}^{-1}(P_{1})\mathcal{G}_{1}^{\prime}(P_{1})=0,\\
P_{1}(T)=0,\\
\Sigma_{0}(P_{1})>0,
\end{array}
\right.
\end{equation}
\begin{equation}\label{RE11}
\left\{
\begin{array}{l}
\tilde{\mathcal{S}}_{1}(P_{1},Q_{1})-\tilde{\mathcal{G}}_{1}(P_{1},Q_{1})\Sigma_{2}(P_{1})^{-1}\tilde{\mathcal{G}}_{1}^{\prime}(P_{1},Q_{1})=0,\\
Q_{1}(T)=0,\\
\Sigma_{2}(P_{1})>0,
\end{array}
\right.
\end{equation}
\begin{equation}\label{RE2}
\left\{
\begin{array}{l}
\mathcal{S}_{2}(P_{2})-\mathcal{G}_{2}(P_{2})\Sigma_{0}^{-1}(P_{2})\mathcal{G}_{2}^{\prime}(P_{2})=0,\\
P_{2}(T)=0,\\
\tilde{\Sigma}_{0}(P_{2})>0,
\end{array}
\right.
\end{equation}
\begin{equation}\label{RE22}
\left\{
\begin{array}{l}
\tilde{\mathcal{S}}_{2}(P_{2},Q_{2})-\tilde{\mathcal{G}}_{2}(P_{2},Q_{2})\Sigma_{2}^{-1}(P_{2})\tilde{\mathcal{G}}_{2}^{\prime}(P_{2},Q_{2})=0,\\
Q_{2}(T)=0,\\
\tilde{\Sigma}_{2}(P_{2})>0,
\end{array}
\right.
\end{equation}
have the solution $(P_{1},Q_{1};P_{2},Q_{2})$ on $[0, T]$. Furthermore, $P_{1}(t),Q_{1}(t) < 0,$ $P_{2}(t),Q_{2}(t)> 0.$
In this case,
\begin{equation*}
\begin{aligned}
&K_2(t) =
 -\tilde{\Sigma}_{0}^{-1}(P_{2})\mathcal{G}_{2}^{\prime}(P_{2}),\\
 &K_{2}(t)+\tilde{K_{2}}(t)=
 - \tilde{\Sigma}_{2}^{-1}(P_{2})\tilde{\mathcal{G}}_{2}^{\prime}(P_{2},Q_{2}),\\
&K_1(t) =
-\Sigma_{0}^{-1}(P_{1})\mathcal{G}_{1}^{\prime}(P_{1}),\\
&K_{1}(t)+\tilde{K_{1}}(t)=
-\Sigma_{2}^{-1}(P_{1}) \tilde{\mathcal{G}}_{1}^{\prime}(P_{1},Q_{1}),\\
&J_{1}(u^{*},v^{*};0,x_{0})=\mathbb{E}\left\langle x_{0}-\mathbb{E}x_{0}, P_{1}(0)(x_{0}-\mathbb{E}x_{0})\right\rangle\\
&+\left\langle\mathbb{E}x_{0},Q_{1}(0)\mathbb{E}x_{0}\right\rangle,\\
&J_{2}\left(u^{*}, v^{*};0,x_{0}\right)=\mathbb{E}\left\langle x_{0}-\mathbb{E}x_{0}, P_{2}(0)(x_{0}-\mathbb{E}x_{0})\right\rangle\\
&+\left\langle\mathbb{E}x_{0},Q_{2}(0)\mathbb{E}x_{0}\right\rangle,
\end{aligned}
\end{equation*}
\end{theorem}
where
\begin{equation*}
\begin{aligned}
\mathcal{S}_{1}(P_{1})&=\dot{P_{1}}+P_{1} (A+B_{2}K_{2})+(A+B_{2}K_{2})^{\prime} P_{1}\\
&+(C+D_{2}K_{2})^{\prime}P_{1} (C+D_{2}K_{2})\\
&+\int_G\{(E+F_{2}K_{2})^{\prime}(\theta)P_{1} (E+F_{2}K_{2})(\theta)\}\nu(d\theta)\\
& -M^{\prime} M-K_{2}^{\prime}K_{2},\\
\mathcal{G}_{1}(P_{1})&=P_{1} B_{1}+(C+D_{2}K_{2})^{\prime} P_{1} D_{1} \\
&+\int_G\{(E+F_{2}K_{2})^{\prime}(\theta)P_{1} F_{1}(\theta)\}\nu(d\theta),\\
\Sigma_{0}(P_{1})&=\gamma^{2}I+D_{1}^{\prime} P_{1}D_{1}+\int_G\{F_{1}^{\prime}(\theta)P_{1} F_{1}(\theta)\}\nu(d\theta),\\
\tilde{\mathcal{S}}_{1}(P_{1},Q_{1})&=\dot{Q_{1}}+Q_{1}[A+\bar{A}+(B_{2}+\bar{B_{2}})(K_{2}+\tilde{K_{2}})]\\
&+\{A+\bar{A}+(B_{2}+\bar{B_{2}})(K_{2}+\tilde{K_{2}})\}^{\prime} Q_{1}\\
&+\{C+\bar{C}+(D_{2}+\bar{D_{2}})(K_{2}+\tilde{K_{2}})\}^{\prime} P_{1}\cdot\\
& [C+\bar{C}+(D_{2}+\bar{D_{2}})(K_{2}+\tilde{K_{2}})] \\
& +\int_G\{[E+\bar{E}+(F_{2}+\bar{F_{2}})(K_{2}+\tilde{K_{2}})]^{\prime}(\theta)P_{1}\cdot\\
&\{E+\bar{E}+(F_{2}+\bar{F_{2}})(K_{2}+\tilde{K_{2}})\}(\theta)\}\nu(d\theta)\\
&-M^{\prime}M-K_{2}^{\prime}K_{2}-\tilde{K}_{2}^{\prime}\tilde{K}_{2},
\end{aligned}
\end{equation*}
\begin{equation*}
\begin{aligned}
\tilde{\mathcal{G}}_{1}(P_{1},Q_{1})&=Q_{1}(B_{1}+\bar{B_{1}})+[C+\bar{C}+(D_{2}+\bar{D_{2}})(K_{2}\\
&+\tilde{K_{2}})]P_{1}(D_{1}+\bar{D_{1}})+\int_G[E+\bar{E}+(F_{2}+\bar{F_{2}})(K_{2}\\
&+\tilde{K_{2}})](\theta)P_{1}(F_{1}+\bar{F_{1}})(\theta)\nu(d\theta),\\
\Sigma_{2}(P_{1})&=\gamma^{2}I+(D_{1}+\bar{D_{1}})^{\prime} P_{1}(D_{1}+\bar{D_{1}})\\
&+\int_G\{(F_{1}+\bar{F_{1}})(\theta)^{\prime}P_{1} (F_{1}+\bar{F_{1}})(\theta)\}\nu(d\theta),
\end{aligned}
\end{equation*}
\begin{equation*}
\begin{aligned}
\mathcal{S}_{2}(P_{2})&=\dot{P_{2}}+P_{2} (A+B_{1}K_{1})+(A+B_{1}K_{1})^{\prime} P_{2}\\
&+(C+D_{1}K_{1})^{\prime} P_{2} (C+D_{1}K_{1}) +\int_G\{(E+F_{1}K_{1})^{\prime}\\
&(\theta)P_{2} (E+F_{1}K_{1})(\theta)\}\nu(d\theta) +M^{\prime} M,
\end{aligned}
\end{equation*}
\begin{equation*}
\begin{aligned}
\mathcal{G}_{2}(P_{2})&=P_{2} B_{2}+(C+D_{1}K_{1})^{\prime} P_{2} D_{2} \\
&+\int_G\{(E+F_{1}K_{1})^{\prime}(\theta)P_{2} F_{2}(\theta)\}\nu(d\theta),\\
\tilde{\Sigma}_{0}(P_{2})&=I+D_{2}^{\prime} P_{2}D_{2}+\int_G\{F_{2}^{\prime}(\theta)P_{2} F_{2}(\theta)\}\nu(d\theta),\\ \nonumber
\end{aligned}
\end{equation*}
\begin{equation*}
\begin{aligned}
\tilde{\mathcal{S}}_{2}(P_{2},Q_{2})&=\dot{Q_{2}}+Q_{2}\{A+\bar{A}+(B_{1}+\bar{B_{1}})(K_{1}+\tilde{K_{1}})\}+\\
&\{A+\bar{A}+(B_{1}+\bar{B_{1}})(K_{1}+\tilde{K_{1}})\}^{\prime} Q_{2}\\
&+\{C+\bar{C}+(D_{1}+\bar{D_{1}})(K_{1}+\tilde{K_{1}})\}^{\prime} P_{2}\cdot\\
& \{C+\bar{C}+(D_{1}+\bar{D_{1}})(K_{1}+\tilde{K_{1}})\} \\
&+\int_G \{(E+\bar{E}+(F_{1}+\bar{F_{1}})(K_{1}+\tilde{K_{1}}))^{\prime}(\theta) P_{2}\cdot \\
& (E+\bar{E}+(F_{1}+\bar{F_{1}})(K_{1}+\tilde{K_{1}}))(\theta)\}\nu(d\theta)\\
& +M^{\prime}M,
\end{aligned}
\end{equation*}
\begin{equation*}
\begin{aligned}
\tilde{\mathcal{G}}_{2}(P_{2},Q_{2})&=Q_{2}(B_{2}+\bar{B_{2}})+[C+\bar{C}+(D_{1}\\
&+\bar{D_{1}})(K_{1}+\tilde{K_{1}})]P_{2}(D_{2}+\bar{D_{2}}) \\ &+\int_G\{E+\bar{E}+(F_{1}+\bar{F_{1}})(K_{1}+\tilde{K_{1}})\}(\theta)\cdot\\
&P_{2}(F_{2}+\bar{F_{2}})(\theta)\nu(d\theta),\\
\end{aligned}
\end{equation*}
\begin{equation*}
\begin{aligned}
\tilde{\Sigma}_{2}(P_{2})=&I+(D_{2}+\bar{D_{2}})^{\prime} P_{2}(D_{2}+\bar{D_{2}})+\\
&\int_G\{(F_{2}+\bar{F_{2}})(\theta)^{\prime}P_{2} (F_{2}+\bar{F_{2}})(\theta)\}\nu(d\theta).
\end{aligned}
\end{equation*}

\begin{proof}
Sufficiency:
From the state equation, we have
 \begin{equation}
\begin{aligned}
&\begin{cases}
d\mathbb{E}[x(t)]=\{(A(t)+\bar{A}(t))\mathbb{E}[x(t)]
+(B_{2}(t)+\bar{B_{2}}(t))\\
\mathbb{E}[u(t)]+(B_{1}(t)+\bar{B_{1}}(t))\mathbb{E}[v(t)]\}dt, \\
\mathbb{E}[x(0)]=\mathbb{E}[x_{0}],
\end{cases}\\
&\begin{cases}
dx(t)-\mathbb{E}[x(t)]=
\{A(t)(x(t)-\mathbb{E}[x(t)])+B_{2}(t)(u(t)-\\\mathbb{E}[u(t)])+
B_{1}(t)(v(t)-\mathbb{E}[v(t)])\}dt\\
+\{C(t)x(t)+\bar{C}(t)\mathbb{E}[x(t)]+D_{2}(t)u(t)+\bar{D_{2}}(t)\mathbb{E}[u(t)]\\
+D_{1}(t)v(t)+\bar{D_{1}}(t)\mathbb{E}[v(t)]\}dW(t)
+\int_G\{E(t,\theta)x(t-)\\
+\bar{E}(t,\theta)\mathbb{E}[x(t-)]+F_{2}(t,\theta)u(t)+\bar{F_{2}}(t,\theta)\mathbb{E}[u(t)]+\\
F_{1}(t,\theta)v(t)+\bar{F_{1}}(t,\theta)\mathbb{E}[v(t)]\}\tilde{N}_p(d\theta,dt),\\
x(0)-\mathbb{E}[x(0)]=x_{0}-\mathbb{E}[x_{0}].
\end{cases}
\end{aligned}
\end{equation}
Then by \mbox{It\^o} formulation and DRE (\ref{RE1}), it follows that (where $x(t-)$ is abbreviated as $x$)
\begin{equation*}
\begin{aligned}
&J_{1}\left(K_{2}x+\tilde{K_{2}}\mathbb{E}(x), v;0,x_{0}\right)\\
=&\mathbb{E} \int_{0}^{T}\left(\gamma^{2} v^{\prime} v-x^{\prime}M^{\prime}Mx-(u^{*})^{\prime}(u^{*})\right) d t\\
&+\mathbb{E} \int_{0}^{T}d((x-\mathbb{E}x)^{\prime}P_{1}(x-\mathbb{E}x))+\mathbb{E} \int_{0}^{T}d((\mathbb{E}x)^{\prime}Q_{1}(\mathbb{E}x))\\
&+\mathbb{E}\left\langle x_{0}-\mathbb{E}x_{0}, P_{1}(0)(x_{0}-\mathbb{E}x_{0})\right\rangle+\left\langle\mathbb{E}x_{0}, Q_{1}(0)\mathbb{E}x_{0}\right\rangle\\
=&\mathbb{E}\left\langle x_{0}-\mathbb{E}x_{0}, P_{1}(0)(x_{0}-\mathbb{E}x_{0})\right\rangle+\left\langle\mathbb{E}x_{0},Q_{1}(0)\mathbb{E}x_{0}\right\rangle\\
&+\mathbb{E} \int_{0}^{T}\bigg\{
\{(v-\mathbb{E}v)+\Sigma_{0}^{-1}(P_{1})\mathcal{G}_{1}^{\prime}(P_{1})(x-\mathbb{E}x)\}^{\prime}\Sigma_{0}(P_{1})\cdot\\
&\{(v-\mathbb{E}v)+\Sigma_{0}^{-1}(P_{1})\mathcal{G}_{1}^{\prime}(P_{1})(x-\mathbb{E}x)\} \\
&+\{(\mathbb{E}v)+\Sigma_{2}^{-1}(P_{1})\tilde{\mathcal{G}}_{1}^{\prime}(P_{1},Q_{1})(\mathbb{E}x)\}^{\prime}\Sigma_{2}(P_{1})\{(\mathbb{E}v)\\
&+\Sigma_{2}^{-1}(P_{1})\tilde{\mathcal{G}}_{1}^{\prime}(P_{1},Q_{1})(\mathbb{E}x)\}  \bigg\}dt.
\end{aligned}
\end{equation*}
It is obvious that $v-\mathbb{E}v=-\Sigma_{0}^{-1}(P_{1})\mathcal{G}_{1}^{\prime}(P_{1})(x-\mathbb{E}x)$ and $\mathbb{E}v=-\Sigma_{2}^{-1}(P_{1})\tilde{\mathcal{G}}_{1}^{\prime}(P_{1},Q_{1})(\mathbb{E}x)$ arrive at the minimum of cost function $J_{1}(u^{*},v;0,x_{0})$, i.e.
\begin{equation*}
\begin{aligned}
v^{*}=&-\Sigma_{0}^{-1}(P_{1})\mathcal{G}_{1}^{\prime}(P_{1})(x(t-)-\mathbb{E}x(t-))\\
&-\Sigma_{2}^{-1}(P_{1})\tilde{\mathcal{G}}_{1}^{\prime}
(P_{1},Q_{1})(\mathbb{E}x(t-)),
\end{aligned}
\end{equation*}
\begin{equation*}
\begin{aligned}
J_{1}(u^{*},v^{*};0,x_{0})=&\mathbb{E}\left\langle x_{0}-\mathbb{E}x_{0}, P_{1}(0)(x_{0}-\mathbb{E}x_{0})\right\rangle\\
&+\left\langle\mathbb{E}x_{0},Q_{1}(0)\mathbb{E}x_{0}\right\rangle.
\end{aligned}
\end{equation*}
When $x_{0}=0,J_{1}(u^{*},v;0,0)\geq 0$, then $\|\mathcal{L}\|\leq \gamma.$
By the same procedure,
\begin{equation*}
\begin{aligned}
&J_{2}(u,v^{*};0,x_{0})\\
=&\mathbb{E}\left\langle x_{0}-\mathbb{E}x_{0}, P_{2}(0)(x_{0}-\mathbb{E}x_{0})\right\rangle+\left\langle\mathbb{E}x_{0},Q_{2}(0)\mathbb{E}x_{0}\right\rangle\\
&+\mathbb{E} \int_{0}^{T}\bigg\{
\{(u-\mathbb{E}u)+\tilde{\Sigma}_{0}^{-1}(P_{2})\mathcal{G}_{2}^{\prime}(P_{2})\}(x-\mathbb{E}x)\}^{\prime}\tilde{\Sigma}_{0}(P_{2})\cdot\\
&\{(u-\mathbb{E}u)+\tilde{\Sigma}_{0}^{-1}(P_{2})\mathcal{G}_{2}^{\prime}(P_{2})\}(x-\mathbb{E}x)\} \\
&+\{(\mathbb{E}u)+\tilde{\Sigma}_{2}^{-1}(P_{2})\tilde{\mathcal{G}}_{2}^{\prime}(P_{2},Q_{2})(\mathbb{E}x)\}^{\prime}\cdot\\
&\tilde{\Sigma}_{2}(P_{2})\{(\mathbb{E}u)+\tilde{\Sigma}_{2}^{-1}(P_{2})\tilde{\mathcal{G}}_{2}^{\prime}(P_{2},Q_{2})(\mathbb{E}x)\}\bigg\}dt,
\end{aligned}
\end{equation*}and
\begin{equation*}
\begin{aligned}
u^{*}=&-\tilde{\Sigma}_{0}^{-1}(P_{2})\mathcal{G}_{2}^{\prime}(P_{2})(x(t-)-\mathbb{E}x(t-))\\
&-\tilde{\Sigma}_{2}^{-1}(P_{2})\tilde{\mathcal{G}}_{2}^{\prime}(P_{2},Q_{2})\mathbb{E}(x(t-)).
\end{aligned}
\end{equation*}
The proof of $\|\mathcal{L}\| < \gamma$ is the same as in MF-SJBRL.

Necessity:
 Implementing $u^{*}(t,x)=K_{2}(t)x(t-)+\tilde{K_{2}}(t)\mathbb{E}(x(t-))$ in system (\ref{xstateequation}), the state equation becomes
\begin{equation*}
\left\{
\begin{aligned}
&dx(t)=\{(A+B_{2}K_{2})x(t)+(\bar{A}+B_{2}\tilde{K_{2}}+\bar{B_{2}}(K_{2}+\tilde{K_{2}}))\cdot\\
&\mathbb{E}[x(t)]
+B_{1}(t)v(t)+\bar{B_{1}}(t)\mathbb{E}[v(t)]\}dt +\\
&\{(C+D_{2}K_{2})x(t)+(\bar{C}+D_{2}\tilde{K_{2}}+\bar{D_{2}}(K_{2}+\tilde{K_{2}}))\mathbb{E}[x(t)]\\
&+D_{1}(t)v(t)+\bar{D_{1}}(t)\mathbb{E}[v(t)]\}dW(t)+\\
&\int_G\{(E+F_{2}K_{2})(\theta)x(t-)+(\bar{E}+F_{2}\tilde{K_{2}}+\bar{F_{2}}(K_{2}+\tilde{K_{2}}))\\
&(\theta)\mathbb{E}[x(t-)]+F_{1}(t,\theta)v(t)+\bar{F_{1}}(t,\theta)\mathbb{E}[v(t)]\}\tilde{N}_p(d\theta,dt),\\
&x(0)=x_{0},\\
&z(t)=\left(
\begin{array}{l}
M x(t) \\
N K_{2}x(t)+N\tilde{K_{2}}\mathbb{E}(x(t))
\end{array}\right).
\end{aligned}
\right.
\end{equation*}
 By definition of $H_{2}/H_{\infty}$ control, we have $\|\mathcal{L}\| < \gamma$. Then we can derive that Riccati equation (\ref{RE1}) has a unique solution $(P,Q)$ through MF-SJBRL. And the worst-case disturbance
 \begin{equation*}
 \begin{aligned}
 v^{*}=&-\Sigma_{0}^{-1}(P_{1})\mathcal{G}_{1}^{\prime}(P_{1})(x(t-)-\mathbb{E}(x(t-)))\\
&-\Sigma_{2}^{-1}(P_{1}) \tilde{\mathcal{G}}_{1}^{\prime}(P_{1},Q_{1})\mathbb{E}(x(t-)).
\end{aligned}
\end{equation*}
 Substituting $v^{*}(t)=K_{1}(t)x(t-)+\tilde{K_{1}}(t)\mathbb{E}(x(t-))$ into system (\ref{xstateequation}), it is obvious that minimizing $J_{2}(u,v^{*};0,x_{0})$ is a classical linear quadratic control problem under standard assumption. By using Theorem \uppercase\expandafter{\romannumeral4}.1. in \cite{2021}, the Riccati equation (\ref{RE2}) has a unique solution. Combining all of the above, we get Theorem \ref{mainresult}. The proof is completed.
\end{proof}

\section{Numerical Simulation}
In this section, we consider a portfolio problem in financial markets. System (\ref{xstateequation}) represents dynamics of the stock price, external interference $v(t)$ represents macroeconomic fluctuations, tariff policy, or other factors on the stock price, jumping process
simulates the instantaneous impact of such as breaking news events or black swan events or other events on the stock price, and the mean field term reflects the interaction between a large number of investors and the market (according to market pricing theory, the game between investors affects the overall price trend through anticipation transmission). So the stock price is modeled as formula (\ref{xstateequation}).

In order to ensure the robustness of the investment strategy, we model it as $H_{2}/H_{\infty}$ control problem represented by (\ref{j1}) and (\ref{j2}). Not only the impact of external interference is considered to avoid reliance on policy intervention leading to a large withdrawal of the portfolio, but also the implementation cost (\ref{j2}) is considered. For continuous coupled Riccati equations, it is not easy to get its unique solution. Therefore, we consider to discretize it and obtain the solution by numerically simulating the difference equation. If the matrix-valued equations (\ref{RE1})-(\ref{RE22}) are solvable, we can obtain $H_{2}/H_{\infty}$ control by the algorithm as
follows:
\begin{enumerate}
  \item For given $\gamma>0$, initial values for $(P,Q)$ and $(\mathcal{G}_1,\tilde{\mathcal{G}}_2,\mathcal{G}_2,\tilde{\mathcal{G}}_2)$ at $T$, we can compute $\Sigma_{0}, \Sigma_{2}, \tilde{\Sigma}_{0}, \tilde{\Sigma}_{2}$ and $K_{1}, K_{1}+\tilde{K}_{1}, K_{2}, K_{2}+\tilde{K}_{2}$.
  \item If $\Sigma_{0}>0, \Sigma_{2}>0, \tilde{\Sigma}_{0}>0, \tilde{\Sigma}_{2}>0$, we can substitute the obtained $K_{1}, K_{1}+\tilde{K}_{1}, K_{2}, K_{2}+\tilde{K}_{2}$ into the matrix equations (\ref{RE1})-(\ref{RE22}). Then $(P_{1}, Q_{1}, P_{2}, Q_{2})$ at $T-\Delta t$ are available by solving the matrix equations (\ref{RE1})-(\ref{RE22}) via backward integration using the right endpoint.
  \item Update $(K_{1}, K_{1}+\tilde{K}_{1}, K_{2}, K_{2}+\tilde{K}_{2})$ and $(\mathcal{G}_1,\tilde{\mathcal{G}}_2,\mathcal{G}_2,\tilde{\mathcal{G}}_2)$ using $(P_{1}, Q_{1}, P_{2}, Q_{2})$, then repeat step (2). The corresponding values at time $t=T, T-\Delta t, T-2\Delta t$, $\cdots, \Delta t, 0 $ can thus be computed recursively.
\end{enumerate}
Next, we present a two-dimensional numerical example. In system (\ref{xstateequation}), set $T=0.1$, $\Delta t=0.001$, $\gamma = 5$, $G=\{1\}, \nu(G)=1$. According to the above algorithm, we can obtain the solutions of the
coupled matrix-valued equations (\ref{RE1}), (\ref{RE11}), (\ref{RE2}), (\ref{RE22}) backward by using standard fourth-order Runge-Kutta iteration procedure. Figure \ref{p1q1} and figure \ref{p2q2} shows the evolution of $P_{1}, Q_{1}$ and $P_{2}, Q_{2}$. Figure \ref{detp1q1} and \ref{detp2q2} shows the evolution of $det(P_{1}), det(Q_{1})$ and $det(P_{2}), det(Q_{2})$.

For simplicity, we set the parameter matrix to be a constant matrix in example. There is no intrinsic difficulty for time-varying matrices. The parameters of system (\ref{xstateequation}) are set as follows:
\begin{equation*}
	\begin{aligned}
&M =\begin{bmatrix} 1 & 0 \\ 0 & 1 \end{bmatrix},
	A = \begin{bmatrix} 1 & 2 \\ -2 & 1 \end{bmatrix},
	\bar{A} = \begin{bmatrix} 1 & -2 \\ 2 & 1 \end{bmatrix},
	B_{1} =\begin{bmatrix} 1  \\  1 \end{bmatrix},
\cr
&\bar{B}_{1} =\begin{bmatrix} 0.5  \\ -1  \end{bmatrix},
	B_{2} =\begin{bmatrix} 1  \\ 1  \end{bmatrix},
	\bar{B}_{2} =\begin{bmatrix} 2  \\ -1  \end{bmatrix},
    C =\begin{bmatrix} 1 & 2 \\ 2 & 1 \end{bmatrix},
\cr
&\bar{C} =\begin{bmatrix} 1 & 2 \\ 2 & 1 \end{bmatrix},
    D_{1} =\begin{bmatrix} 2  \\  1 \end{bmatrix},
	\bar{D}_{1} =\begin{bmatrix} 1 \\ 1 \end{bmatrix},
  D_{2} =\begin{bmatrix} 2  \\ -2  \end{bmatrix},
  \cr
 & 	\bar{D}_{2} =\begin{bmatrix} -2  \\  2 \end{bmatrix},
    E =\theta*\begin{bmatrix} -1 & 1 \\ 3 & 1 \end{bmatrix},
    \bar{E} =\theta*\begin{bmatrix} -1 & 0 \\ 3 & 1 \end{bmatrix},
 \cr
&   F_{1} =\theta*\begin{bmatrix} 2  \\ 1 \end{bmatrix},
\bar{F}_{1} =\theta*\begin{bmatrix} 2  \\  2 \end{bmatrix},
    F_{2} =\theta*\begin{bmatrix} 2  \\  1 \end{bmatrix},
    \bar{F}_{2} =\theta*\begin{bmatrix} 2  \\ 2 \end{bmatrix}.
\end{aligned}
\end{equation*}
If we keep reducing the value of $\gamma$, we will meet a threshold which determine the solvability of Riccati equations (\ref{RE1})-(\ref{RE22}).
\begin{figure}[!b]
 \centering
  \centerline{\includegraphics[width=\columnwidth]{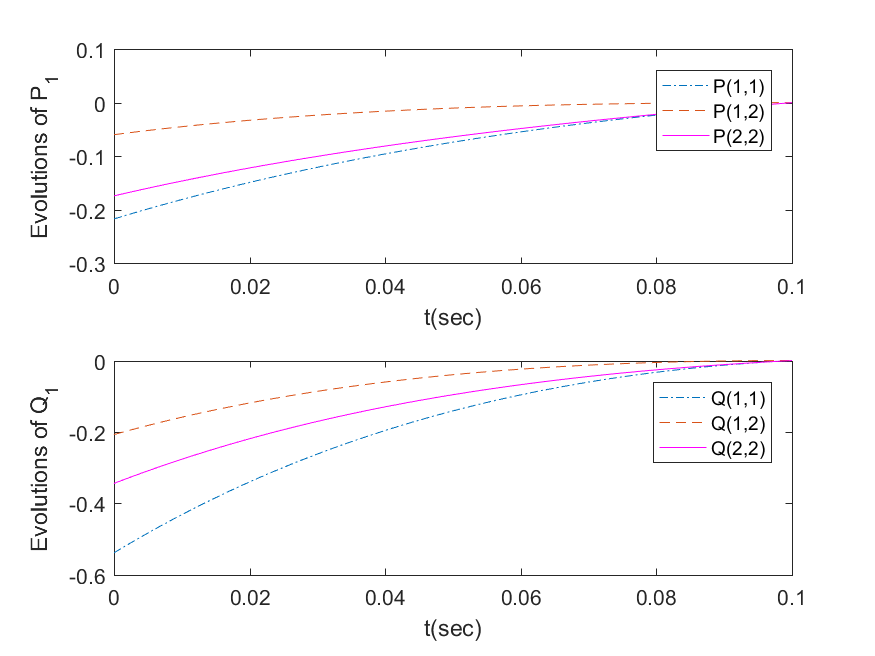}}
 \caption{The trajectories of $P_{1}$ and $Q_{1}$}\label{p1q1}
\end{figure}
\begin{figure}[!b]
 \centering
  \centerline{\includegraphics[width=\columnwidth]{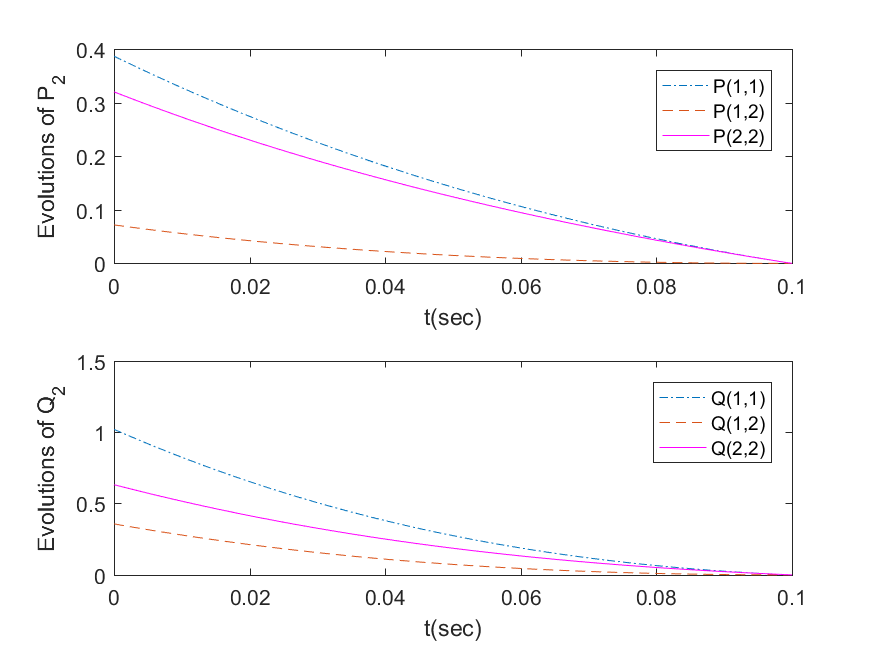}}
 \caption{The trajectories of $P_{2}$ and $Q_{2}$}\label{p2q2}
\end{figure}
\begin{figure}
 \centering
  \centerline{\includegraphics[width=\columnwidth]{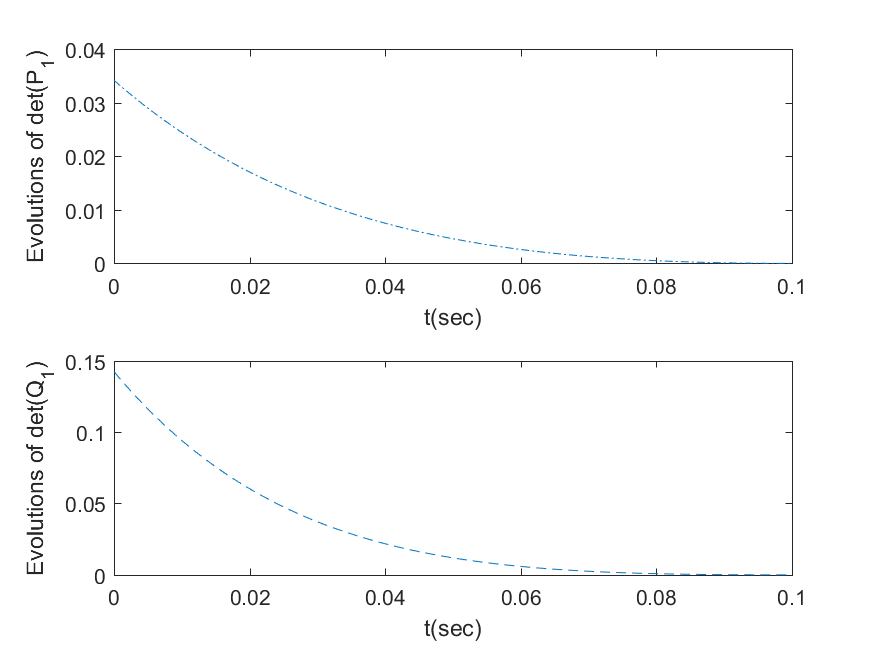}}
 \caption{The trajectories of $det(P_{1})$ and $det(Q_{1})$}\label{detp1q1}
\end{figure}
\begin{figure}
 \centering
  \centerline{\includegraphics[width=\columnwidth]{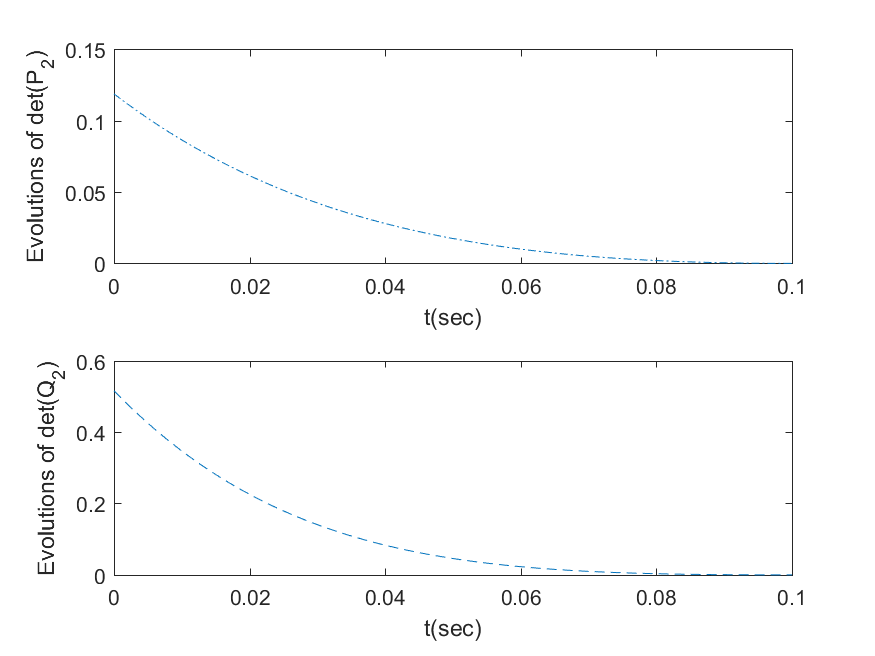}}
 \caption{The trajectories of $det(P_{2})$ and $det(Q_{2})$}\label{detp2q2}
\end{figure}
\begin{figure}
 \centering
 \centerline{\includegraphics[width=\columnwidth]{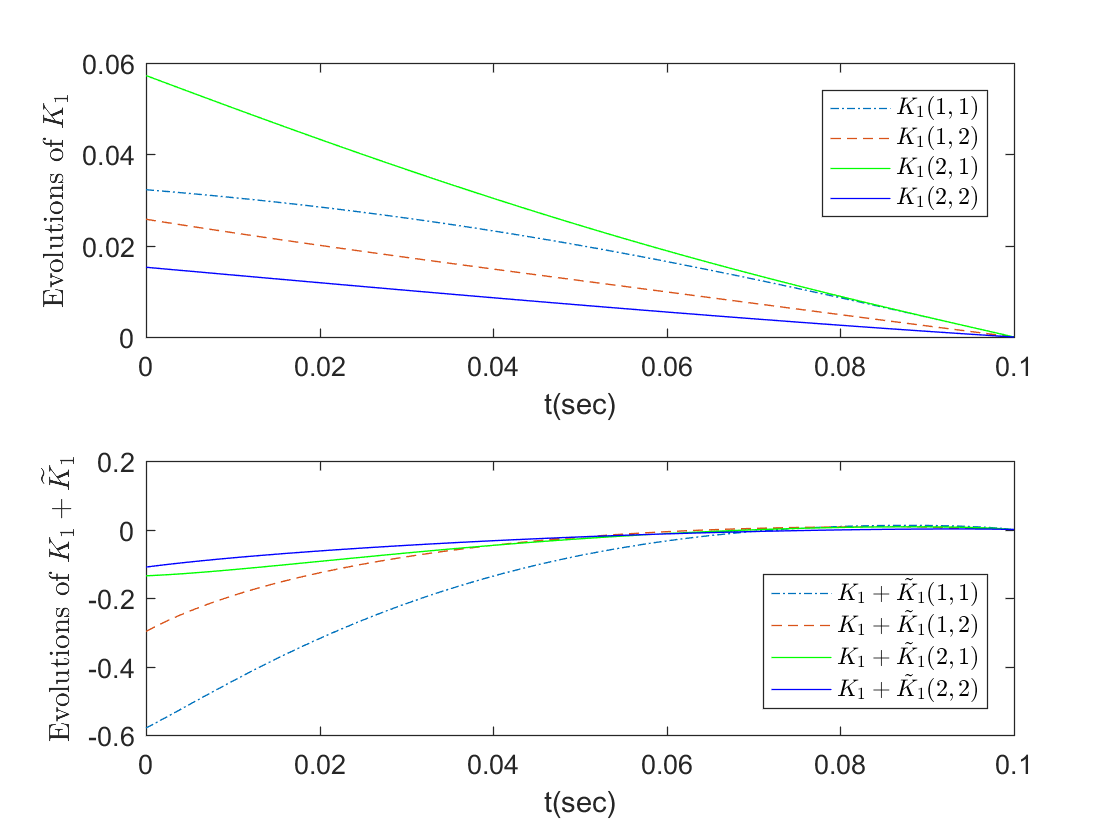}}
 \caption{The trajectories of $K_{1}$ and $K_{1}+\tilde{K}_{1}$}\label{k1}
\end{figure}
\begin{figure}
\centering
\centerline{\includegraphics[width=\columnwidth]{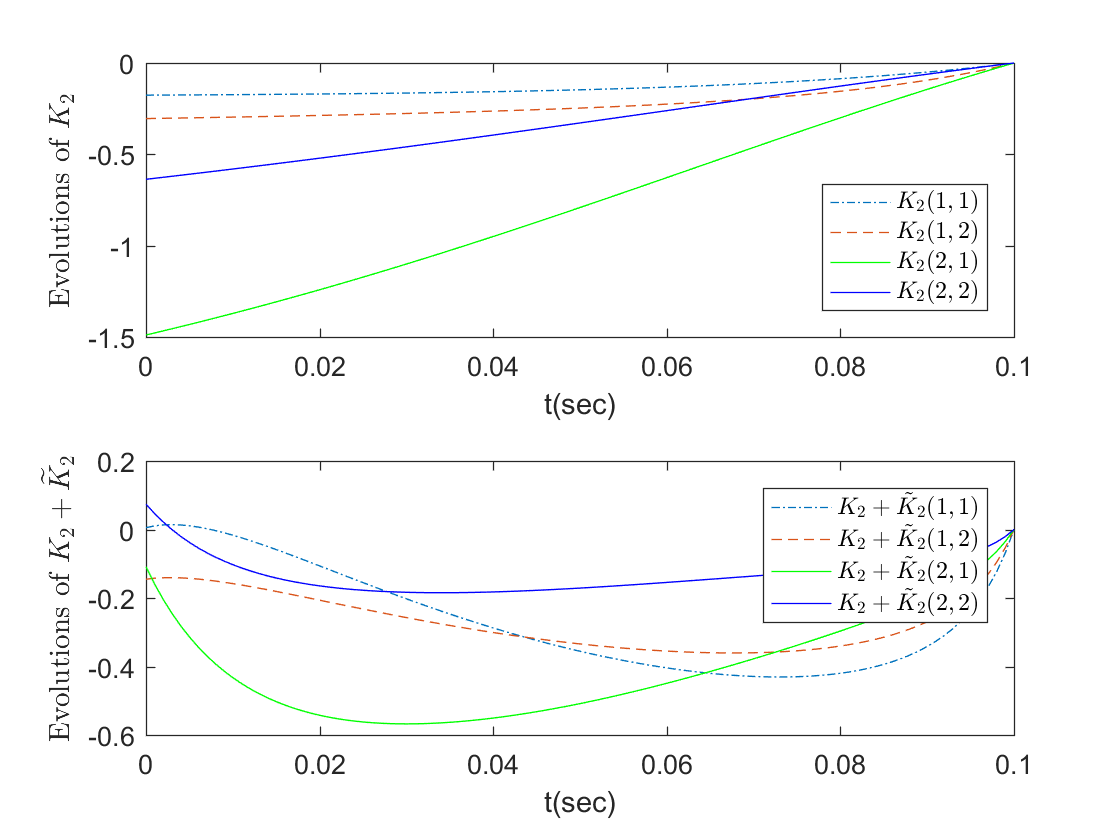}}
 \caption{The trajectories of $K_{2}$ and $K_{2}+\tilde{K}_{2}$}\label{k2}
\end{figure}

\section{Conclusion}
This paper discussed the finite horizon $H_2/H_\infty$ control problem for mean-field jump systems with $(x, u, v)$-dependent noise. A necessary and sufficient condition is derived based on four coupled Riccati equations, for which a recursive algorithm is provided. A model-free reinforcement learning approach is also proposed to design robust controllers for mean-field systems. Potential extensions include applying the framework to infinite horizon problems and systems with random coefficients.

\section*{Appendixes}
1. To facilitate readers' understanding and avoid potential misinterpretations, we first present a proof sketch of MF-SJBRL.

Sufficiency:
\begin{itemize}
  \item Complete the square for $J_1(0,v,\tau,\xi)$ using equations (\ref{3.4})-(\ref{3.5}) to obtain $J_1(0,v,\tau,\xi) \geq 0$.
  \item Prove $J_1(0,v,\tau,\xi) > 0, \forall v\neq0$ via the inverse mapping theorem.
\end{itemize}

Necessity:
\begin{itemize}
 \item  Derive the quasi-linear equation (\ref{quasi}) from (\ref{3.4}).

 \item  Perform Picard iteration for any initial matrix $\hat{P}$ using (\ref{quasi}).

 \item  Apply Lemma 3 to show that the sequence $\{P_n\}$ generated by the Picard iteration is monotonic.

 \item  Use Lemma 5 and 6 to prove that the decreasing sequence $\{P_n\}$ obtained from the Picard iteration is bounded below; then apply the monotone convergence theorem and the dominated convergence theorem to prove that the sequence has a limit and the limit is solution to (\ref{3.4}).

 \item  It follows from Lemma 7 that the algebraic condition $\Sigma_0(P)>0$ is satisfied.

 \item  Repeat the above process for equation (\ref{3.5}).
\end{itemize}

\begin{proof}[Proof of Lemma \ref{lemmaPpositive}]
Since the equation is linear and all coefficients are uniformly bounded, it admits a unique solution $P \in C\left(0, T ; S^{n}\right)$. For any given $x \in \mathbb{R}^{n}$, suppose $\phi(\cdot)$ is the solution of the following equation:
\begin{equation*}
\left\{
\begin{aligned}
\mathrm{d} \phi(s)=&\tilde{A}(s) \phi(s) ds+\tilde{C}(s) \phi(s) \mathrm{d} W(s)\\
&+\int_{G}\left(\tilde{E}(s) \phi(s-)\right) \nu (d\theta)ds, \\
\phi(t)=&x, \quad t \in[0, T] .
\end{aligned}
\right.
\end{equation*}
Through \mbox{It\^o} formula, we obtain
\begin{equation}\label{lemmap}
\begin{aligned}
&d\left(\phi^{\prime}(s) P(s) \phi(s)\right)=\phi(s)^{\prime}\dot{P}(s)\phi(s)ds+\\
&\phi^{\prime}(s)\left[P(s) \tilde{A}(s)+\tilde{A}^{\prime}(s) P(s)+\tilde{C}^{\prime}(s) P(s) \tilde{C}(s)\right] \phi(s) ds \\
&+\phi^{\prime}(s)\left[\tilde{C}^{\prime}(s) P(s)+P(s) \tilde{C}(s)\right] \phi(s) \mathrm{d} W(s) \\ &+\int_{G}\left[\phi^{\prime}(s)\tilde{E}^{\prime}(s) P(s) \tilde{E}(s) \phi(s)\right] \nu(d\theta)ds \\
&+\int_{G}\phi^{\prime}(s-)\tilde{E}^{\prime}(s) P(s) \tilde{E}(s) \phi(s-) \tilde{N}_{p}(d\theta,ds) .
\end{aligned}
\end{equation}
Integrating from $t$ to $T$, and taking $\mathbb{E}$ on both sides of (\ref{lemmap}) yield
\begin{equation*}
\langle P(t) x, x\rangle=\mathbb{E}\left\{\phi^{\prime}(T) \tilde{G} \phi(T)+\int_{t}^{T} \phi^{\prime}(s) \tilde{Q}(s) \phi(s) \mathrm{d} s\right\}.
\end{equation*}
Given $\tilde{G} \geqslant 0, \tilde{Q} \geqslant 0$,  it follows that $P(t) \geqslant 0$ for all $t\in[0,T]$.
\end{proof}

\begin{proof}[Proof of Lemma \ref{lemmalinearj1}]
\begin{equation*}
\begin{aligned}
& J_{1}\left(0,v+\varphi(x^{\varphi,\psi}-\mathbb{E}x^{\varphi,\psi})+ \psi\mathbb{E}x^{\varphi,\psi}; \tau, \xi\right) \\
=&\mathbb{E} \int_{\tau}^{T}\left(\gamma^{2}\|v+\varphi(x^{\varphi,\psi}-\mathbb{E}x^{\varphi,\psi})+\psi\mathbb{E}x^{\varphi,\psi}\|^{2}\right.\\
&\left.-(x^{\varphi,\psi})^{\prime}M_{11}^{\prime} M_{11} x^{\varphi,\psi}\right) dt\\
&+\mathbb{E} \int_{\tau}^{T}d((x^{\varphi,\psi}-\mathbb{E}x^{\varphi,\psi})^{\prime}P^{\gamma, \varphi}(x^{\varphi,\psi}-\mathbb{E}x^{\varphi,\psi}))\\
&-\mathbb{E} \int_{\tau}^{T}d((x^{\varphi,\psi}-\mathbb{E}x^{\varphi,\psi})^{\prime}P^{\gamma, \varphi}(x^{\varphi,\psi}-\mathbb{E}x^{\varphi,\psi}))\\
&+\mathbb{E} \int_{\tau}^{T}d((\mathbb{E}x^{\varphi,\psi})^{\prime}Q^{\gamma, \varphi,\psi}(\mathbb{E}x^{\varphi,\psi}))\\
&-E \int_{\tau}^{T}d((\mathbb{E}x^{\varphi,\psi})^{\prime}Q^{\gamma, \varphi,\psi}(\mathbb{E}x^{\varphi,\psi})) .
\end{aligned}
\end{equation*}
Then by Lemma \ref{lemmasemiito} and equations (\ref{SBRLlinearRE1})-(\ref{SBRLlinearRE2}), we can derive (\ref{SBRLJ1linear}).
\end{proof}

\begin{proof}[Proof of Lemma \ref{lemmainitial}]
By linearity of the system, the solution $x(t,v;\tau,\xi)$ to system (\ref{SBRLxstate}) can be decomposed as
$x(t,v;\tau,\xi)=x(t,v;\tau,0)+x(t,0;\tau,\xi).$
Denote $X$ and $Y$ as the solutions of
\begin{equation*}
\begin{aligned}
&\begin{cases}
  \mathcal{S}(X)=0,\\
  X(T)=0
\end{cases}\\
\end{aligned}
\end{equation*}
and
\begin{equation*}
\begin{aligned}
&\begin{cases}
    \tilde{\mathcal{S}}(Y,Y)=0,\\
    Y(T)=0
\end{cases}
\end{aligned}
\end{equation*}
respectively. It is easy to check that
\begin{equation*}
\begin{aligned}
&J_{1}\left(0,v; \tau, \xi\right)-J_{1}\left(0,v; \tau, 0\right)\\
=&\mathbb{E}\left\langle(\xi-\mathbb{E}\xi), X_{\tau}(\xi-\mathbb{E}\xi)\right\rangle+\left\langle\mathbb{E}\xi, Y_{\tau}\mathbb{E}\xi\right\rangle\\
+&\mathbb{E}\int_{\tau}^{T}(v-\mathbb{E}v)\mathcal{G}^{\prime}(X)(x(t,0;\tau,\xi)-\mathbb{E}x(t,0;\tau,\xi))\\
+&(x(t,0;\tau,\xi)-\mathbb{E}x(t,0;\tau,\xi))^{\prime}\mathcal{G}(X)(v-\mathbb{E}v)\\
+&(\mathbb{E}v)\tilde{\mathcal{G}}^{\prime}(X,Y)(\mathbb{E}x(t,0;\tau,\xi))\\
+&(\mathbb{E}x(t,0;\tau,\xi))^{\prime}\tilde{\mathcal{G}}(X,Y)(\mathbb{E}v)dt.
\end{aligned}
\end{equation*}
Because of $\|\tilde{\mathcal{L}}\|<\gamma$, we can take $0 \leq \epsilon^{2} \leq \gamma^{2}-\|\tilde{\mathcal{L}}\|^{2}$, then
\begin{equation*}
\begin{aligned}
&J_{1}\left(0,v; \tau, 0\right)\geq  \gamma^{2}\|\bar{v}\|^{2}_{[0,T]}-\|z_{1}\|^{2}_{[0,T]}\\
\geq & (\gamma^{2}-\|\tilde{\mathcal{L}}\|^{2})\|\bar{v}\|^{2}_{[0,T]}
\geq  \epsilon^{2}\|\bar{v}\|^{2}_{[0,T]}
= \epsilon^{2}\|v\|^{2}_{[\tau,T]},
\end{aligned}
\end{equation*}
where
$$\bar{v}=
\begin{cases}
v, \quad t\in[\tau,T],\\
0, \quad t\in[0,\tau) .
\end{cases}$$
Therefore, by completing the square,
\begin{equation*}
\begin{aligned}
&J_{1}\left(0,v; \tau, \xi\right)\\
\geq&\mathbb{E}\left\langle(\xi-\mathbb{E}\xi), X_{\tau}(\xi-\mathbb{E}\xi)\right\rangle+\left\langle\mathbb{E}\xi, Y_{\tau}\mathbb{E}\xi\right\rangle+\mathbb{E}\int_{\tau}^{T}\Big\{\epsilon^{2}\|v\|^{2}\\
&+(v-\mathbb{E}v)\mathcal{G}^{\prime}(X)(x(t,0;\tau,\xi)-\mathbb{E}x(t,0;\tau,\xi))\\
&+(x(t,0;\tau,\xi)-\mathbb{E}x(t,0;\tau,\xi))^{\prime}\mathcal{G}(X)(v-\mathbb{E}v)\\
&+(\mathbb{E}v)\tilde{\mathcal{G}}^{\prime}(X,Y)(\mathbb{E}x(t,0;\tau,\xi))\\
&+(\mathbb{E}x(t,0;\tau,\xi))^{\prime
}\tilde{\mathcal{G}}(X,Y)(\mathbb{E}v)\Big\}dt\\
\geq&\mathbb{E}\left\langle(\xi-\mathbb{E}\xi), X_{\tau}(\xi-\mathbb{E}\xi)\right\rangle+\left\langle\mathbb{E}\xi, Y_{\tau}\mathbb{E}\xi\right\rangle\\
&-\mathbb{E}\int_{\tau}^{T}\|\frac{1}{\epsilon}\mathcal{G}^{\prime}(X)(x(t,0;\tau,\xi)-\mathbb{E}x(t,0;\tau,\xi))\|^{2}\\
&-\|\frac{1}{\epsilon}\tilde{\mathcal{G}}^{\prime}(X,Y)(\mathbb{E}x(t,0;\tau,\xi))\|^{2}dt.
\end{aligned}
\end{equation*}
By Lemma \ref{lemmaxstatesolution} and the estimate (\ref{estimate}), there are $\alpha_{1},\alpha_{2} >0$ satisfying
\begin{equation*}
\begin{aligned}
\mathbb{E}\int_{\tau}^{T}\|(x(t,0;\tau,\xi)-\mathbb{E}x(t,0;\tau,\xi))\|^{2}dt &\leq \alpha_{1}\mathbb{E}\|\xi-\mathbb{E}\xi\|^{2},\\
\mathbb{E}\int_{\tau}^{T}\|\mathbb{E}x(t,0;\tau,\xi)\|^{2}dt &\leq \alpha_{2}\|\mathbb{E}\xi\|^{2},
\end{aligned}
\end{equation*}
and there are $\alpha_{3},\alpha_{4}>0$ that the following hold.
\begin{equation*}
\begin{aligned}
&\mathbb{E}\left\langle(\xi-\mathbb{E}\xi), X_{\tau}(\xi-\mathbb{E}\xi)\right\rangle=-\mathbb{E}\int_{\tau}^{T}d(x(t,0;\tau,\xi)-\\
&\mathbb{E}x(t,0;\tau,\xi))^{\prime}X (x(t,0;\tau,\xi)-\mathbb{E}x(t,0;\tau,\xi))\\
=&-\mathbb{E}\int_{\tau}^{T}(x(t,0;\tau,\xi)-\mathbb{E}x(t,0;\tau,\xi))^{\prime}M_{11}^{\prime}M_{11}(x(t,0;\tau,\xi)\\
&-\mathbb{E}x(t,0;\tau,\xi))dt\\
\geq&-\alpha_{3}\mathbb{E}\|\xi-\mathbb{E}\xi\|^{2},\\
&\left\langle\mathbb{E}\xi, Y_{\tau}\mathbb{E}\xi\right\rangle=-\mathbb{E}\int_{\tau}^{T}d(\mathbb{E}x)^{\prime}Y (\mathbb{E}x)\\
=&-\mathbb{E}\int_{\tau}^{T}(\mathbb{E}x(t,0;\tau,\xi))^{\prime}M_{11}^{\prime}M_{11}(\mathbb{E}x(t,0;\tau,\xi))dt\\
\geq&-\alpha_{4}\|\mathbb{E}\xi\|^{2}.
\end{aligned}
\end{equation*}
Then there exists $\mu> 0$, such that $J_{1}\left(0, v; \tau, \xi\right)\geq -\mu \mathbb{E}|\xi|^{2}.$ The proof is completed.
\end{proof}

\begin{proof}[Proof of Lemma \ref{main}]
For any deterministic $\tilde{v}(\cdot) \in \mathbb{R}^{n_{v}}$, let $x$ be the solution of
\begin{equation*}
\left\{
\begin{aligned}
dx(t)=&\big\{A_{11}(t)x(t)+B_{11}(t)v(t)\big\}dt\\
&+\big\{C_{11}(t)x(t)+D_{11}(t)v(t)\big\}dW(t)\\
&+\int_G\{E_{11}(t,\theta)x(t-)+F_{11}(t,\theta)v(t)\}\tilde{N}_p(d\theta,dt), \\
x(0)=&0, \quad t \in[0, T],\\
\end{aligned}
\right.
\end{equation*}
and set (where $t-$ is omitted and will not be noted hereafter).
\begin{equation*}
v(\cdot) \triangleq \tilde{v} W+\varphi(x-\mathbb{E}x)+ \psi\mathbb{E}x \quad \in \mathcal{U}([0,T];\mathbb{R}^{n_{v}}) .
\end{equation*}
Clearly,
$$
\mathbb{E}\left[x(t)\right]=0, \quad \mathbb{E}[v(t)]=0, \quad t \in[0, T]
$$
By the uniqueness of the solution, $x$ also solves (\ref{SBRLxstate}) when $x_{0}=0$.

If $\|\tilde{\mathcal{L}}\|<\gamma$, then
$$
J_{1}\left(0, v; 0, 0\right)\geq \delta \mathbb{E} \int_{0}^{T}|v(s)|^{2} ds, \forall v\in \mathcal{U}([0,T];\mathbb{R}^{n_{v}}).
$$
By Lemma \ref{lemmalinearj1},
\begin{equation}\label{11}
\begin{aligned}
& J_{1}\left(0,\tilde{v}W+\varphi(x^{\varphi,\psi}-\mathbb{E}x^{\varphi,\psi})+ \psi\mathbb{E}x^{\varphi,\psi}; 0, 0\right) \\
=&\mathbb{E}
\int_{0}^{T}
\left\langle \tilde{v}W,(\mathcal{G}^{\prime}(P^{\gamma, \varphi})+\Sigma_{0}(P^{\gamma, \varphi})\varphi(t))x^{\varphi,\psi}(t)\right\rangle\\
&+\left\langle(\mathcal{G} ^{\prime}(P^{\gamma, \varphi})+\Sigma_{0}(P^{\gamma, \varphi})\varphi(t))x^{\varphi,\psi}(t),\tilde{v}W\right\rangle\\
&+\left\langle \tilde{v}W, \Sigma_{0}(P^{\gamma, \varphi})\tilde{v}W \right\rangle dt  \\
\geq&\delta \mathbb{E} \int_{0}^{T}|\tilde{v}W+\varphi x^{\varphi,\psi}|^{2} dt.
\end{aligned}
\end{equation}
Hence, the following holds:
\begin{equation*}
\begin{aligned}
&\mathbb{E} \int_{0}^{T} 2\left\langle\left[\mathcal{G}^{\prime}(P^{\gamma, \varphi})+(\Sigma_{0}(P^{\gamma, \varphi})-\delta I)\varphi\right] W x^{\varphi,\psi}, \tilde{v}\right\rangle\\
& +W^2\left\langle\left(\Sigma_{0}(P^{\gamma, \varphi})-\delta I\right) \tilde{v}, \tilde{v}\right\rangle dt\geq 0 .
\end{aligned}
\end{equation*}
Now, applying \mbox{It\^o's} formula, we have
\begin{equation*}
\left\{\begin{array}{l}
d \mathbb{E}\left[W(s) x^{\varphi,\psi}(s)\right]=\left\{[(A_{11}(s)+B_{11}(s)\varphi(s))]\right. \\
\left.\mathbb{E}\left[W(s) x^{\varphi,\psi}(s)\right]+s B_{11}(s) \tilde{v}(s)\right\} \mathrm{d} s, \quad s \in[0, T], \\
\mathbb{E}\left[W(0) x^{\varphi,\psi}(0)\right]=0 .
\end{array}\right.
\end{equation*}
Fix any $u_0 \in \mathbb{R}^{n_{v}}$ and take $\tilde{v}(s)=u_0 \mathbf{1}_{\left[t^{\prime}, t^{\prime}+h\right]}(s)$, with $0 < t^{\prime}<t^{\prime}+h \leqslant T$. Then
\begin{equation*}
\mathbb{E}\left[W(s) x(s)\right]= \begin{cases}
0, \quad\quad s \in\left[0, t^{\prime}\right], \\
\Phi(s) \int_{t^{\prime}}^{s \wedge\left(t^{\prime}+h\right)} \Phi(r)^{-1} B_{11}(r) r u_0 \mathrm{~d} r, \\
\quad\quad\quad s \in\left[t^{\prime}, T\right],
\end{cases}
\end{equation*}
where $\Phi(\cdot)$ is the solution of the following ordinary differential equation:
\begin{equation*}
\left\{\begin{array}{l}
\dot{\Phi}(s)=[A_{11}(s)+B_{11}(s)\varphi(s)] \Phi(s), \quad s \in[0, T], \\
\Phi(0)=I .
\end{array}\right.
\end{equation*}
Consequently, (\ref{11}) becomes
\begin{equation*}
\begin{aligned}
\int_{t^{\prime}}^{t^{\prime}+h}\Big\{ & 2\left\langle\left[\mathcal{G}^{\prime}(P^{\gamma, \varphi})+\left(\Sigma_{0}(P^{\gamma, \varphi})-\delta I\right) \varphi\right] \Phi(s)\cdot\right.\\
&\left. \int_{t^{\prime}}^s \Phi(r)^{-1} B_{11}(r) r u_0 \mathrm{d} r, u_0\right\rangle \\
&+s\left\langle\left(\Sigma_{0}(P^{\gamma, \varphi})-\delta I\right) u_0, u_0\right\rangle\Big\} \mathrm{d} s \geqslant 0 .
\end{aligned}
\end{equation*}
Dividing both sides by $h$ and letting $h \rightarrow 0$, by using Lebesgue differentiation theorem, we obtain
\begin{equation*}
t^{\prime}\left\langle\left[\Sigma_{0}(P^{\gamma, \varphi})-\delta I\right] u_0, u_0\right\rangle \geqslant 0, \quad \forall u_0 \in \mathbb{R}^{n_{v}},  t^{\prime} \in(0, T].
\end{equation*}
By the continuity of $\Sigma_{0}(P^{\gamma, \varphi})$ on [0,T], $\Sigma_{0}(P^{\gamma, \varphi})\geq \delta I$ .
Set\begin{equation*}
v(\cdot) \triangleq \tilde{v}+\varphi(x-\mathbb{E}x)+ \psi\mathbb{E}x \quad \in \mathcal{U}([0,T];\mathbb{R}^{n_{v}}) .
\end{equation*}
By lemma \ref{lemmalinearj1},
\begin{equation}\label{22}
\begin{aligned}
& J_{1}\left(0,\tilde{v}+\varphi(x^{\varphi,\psi}-\mathbb{E}x^{\varphi,\psi})+ \psi\mathbb{E}x^{\varphi,\psi}; 0, 0\right)= \\
&\mathbb{E}
\int_{0}^{T}\left\langle\mathbb{E}v(t),(\tilde{\mathcal{G}}^{\prime}(P^{\gamma, \varphi},Q^{\gamma,\varphi,\psi})+\Sigma_{2}(P^{\gamma, \varphi})\psi(t))\mathbb{E}x^{\varphi,\psi}(t)\right\rangle\\
&+\left\langle(\tilde{\mathcal{G}}^{\prime}(P^{\gamma, \varphi},Q^{\gamma,\varphi,\psi})+\Sigma_{2}(P^{\gamma, \varphi})\psi(t))\mathbb{E}x^{\varphi,\psi}(t),\mathbb{E}v(t)\right\rangle\\
&+\left\langle\mathbb{E}v(t), \Sigma_{2}(P^{\gamma, \varphi})\mathbb{E}v(t)\right\rangle dt\\
\geq &\delta \mathbb{E} \int_{0}^{T}|\tilde{v}+\varphi(x^{\varphi,\psi}-\mathbb{E}x^{\varphi,\psi})+ \psi\mathbb{E}x^{\varphi,\psi}|^{2} dt.
\end{aligned}
\end{equation}
Hence, the following holds:
\begin{equation*}
\begin{aligned}
 &\int_{0}^{T} 2\left\langle\left[\tilde{\mathcal{G}}^{\prime}(P^{\gamma, \varphi},Q^{\gamma,\varphi,\psi})+(\Sigma_{2}(P^{\gamma, \varphi})-\delta I)\psi(t)\right] \mathbb{E} x^{\varphi,\psi}, \tilde{v}\right\rangle\\
 & +\left\langle\left(\Sigma_{2}(P^{\gamma, \varphi})-\delta I\right) \tilde{v}, \tilde{v}\right\rangle dt\geq 0 .
\end{aligned}
\end{equation*}
Now, applying \mbox{It\^o} formula, we have
\begin{equation*}
\left\{\begin{array}{l}
d \mathbb{E}\left[x^{\varphi,\psi}(s)\right]=\big\{[(A_{11}(s)+\bar{A}_{11}(s))+(B_{11}(s)\\
+\bar{B}_{11}(s))\psi(s)] \mathbb{E}\left[ x^{\varphi,\psi}(s)\right]+(B_{11}(s)+\bar{B}_{11}(s))\tilde{v}(s)\big\} \mathrm{d} s,  \\
\mathbb{E}\left[x^{\varphi,\psi}(0)\right]=0 ,\quad s \in[0, T].
\end{array}\right.
\end{equation*}
Fix any $u_0 \in \mathbb{R}^{n_{v}}$ and take $\tilde{v}(s)=u_0 \mathbf{1}_{\left[t^{\prime}, t^{\prime}+h\right]}(s)$, with $0 \leqslant t^{\prime}<t^{\prime}+h \leqslant T$. Then
\begin{equation*}
\mathbb{E}\left[x^{\varphi,\psi}(s)\right]=
\begin{cases}
0, \quad \quad s \in\left[0, t^{\prime}\right], \\
\Phi(s) \int_{t^{\prime}}^{s \wedge\left(t^{\prime}+h\right)} \Phi(r)^{-1} \bar{B}_{11}(r) u_0 \mathrm{d} r,\\
\quad \quad \quad s \in\left[t^{\prime}, T\right],
\end{cases}
\end{equation*}
where $\Phi(\cdot)$ is the solution of the following ordinary differential equation:
\begin{equation*}
\left\{\begin{array}{l}
\dot{\Phi}(s)=[(A_{11}(s)+\bar{A}_{11}(s))+(B_{11}(s)+\bar{B}_{11}(s))\psi(s)] \Phi(s), \\
\Phi(0)=I,  \quad s \in[0, T].
\end{array}\right.
\end{equation*}
Consequently, (\ref{22}) becomes
\begin{equation*}
\begin{aligned}
\int_{t^{\prime}}^{t^{\prime}+h}\Big\{& 2\left\langle\left[\tilde{\mathcal{G}}^{\prime}(P^{\gamma, \varphi},Q^{\gamma,\varphi,\psi})+(\Sigma_{2}(P^{\gamma, \varphi})-\delta I)\psi(s)\right] \Phi(s)\cdot\right.\\
&\left. \int_{t^{\prime}}^s \Phi(r)^{-1} \bar{B}_{11}(r) u_0 \mathrm{d} r, u_0\right\rangle \\
&+\left\langle\left(\Sigma_{2}(P^{\gamma, \varphi})-\delta I\right) u_0, u_0\right\rangle\Big\} \mathrm{d} s \geqslant 0 .
\end{aligned}
\end{equation*}
Dividing both sides by $h$ and letting $h \rightarrow 0$, by using Lebesgue differentiation theorem, we obtain
\begin{equation*}
\left\langle\left[\Sigma_{2}(P^{\gamma, \varphi})-\delta I\right] u_0, u_0\right\rangle \geqslant 0, \quad \forall u_0 \in \mathbb{R}^{n_{v}}, \quad \text { a.e. } t^{\prime} \in[0, T].
\end{equation*}
So $\Sigma_{0}(P^{\gamma, \varphi})\geq \delta I$ and $\Sigma_{2}(P^{\gamma, \varphi})\geq \delta I$.
\end{proof}

\begin{proof}[Proof of Lemma \ref{SBJRL}]
Sufficiency:

By \mbox{It\^o} formulation and DRE (\ref{3.4})-(\ref{3.5}), the following equation holds.
\begin{equation*}
\begin{aligned}
& J_{1}\left(0,v; \tau, \xi\right) \\
=&\mathbb{E}\left\langle \xi-\mathbb{E}\xi, P(\tau)(\xi-\mathbb{E}\xi)\right\rangle+\left\langle\mathbb{E}\xi,Q(\tau)\mathbb{E}\xi\right\rangle\\
&+\mathbb{E} \int_{\tau}^{T}\big\{
\{(v-\mathbb{E}v)-\Phi(P)(x-\mathbb{E}x)\}^{\prime}\Sigma_{0}(P)\\
&\{(v-\mathbb{E}v)-\Phi(P)(x-\mathbb{E}x)\} \\
&+\{(\mathbb{E}v)-\Psi(P,Q)(\mathbb{E}x)\}^{\prime}\Sigma_{2}(P)\\
&\{(\mathbb{E}v)-\Psi(P,Q)(\mathbb{E}x)\}  \big\}dt.
\end{aligned}
\end{equation*}
When $\tau=0,\xi=0$, $$J_{1}\left(0,v; 0, 0\right)=\mathbb{E} \int_{0}^{T}\left(\gamma^{2}\|v\|^{2}-\left\|z_{1}\right\|^{2}\right) d t \geq 0,$$
i.e. $\|\tilde{\mathcal{L}}\|_{[0,T]}\leq \gamma$. We prove $\|\tilde{\mathcal{L}}\|_{[0,T]} < \gamma$ below.
Define the operators $\mathcal{L}_{1}: L_{\mathcal{F}}^{2}\left([0, T], \mathbb{R}^{n_{v}}\right) \mapsto L_{\mathcal{F}}^{2}\left([0, T], \mathbb{R}^{n_{v}}\right)$ and $\tilde{\mathcal{L}}_{1}: L_{\mathcal{F}}^{2}\left([0, T], \mathbb{R}^{n_{v}}\right) \mapsto L_{\mathcal{F}}^{2}\left([0, T], \mathbb{R}^{n_{v}}\right)$ as
$$
\mathcal{L}_{1} (v(t)-\mathbb{E}v(t))=v(t)-\mathbb{E}v(t)-(v^{*}(t)-\mathbb{E}v^{*}(t)),
$$
$$
\tilde{\mathcal{L}}_{1} (\mathbb{E}v(t))=\mathbb{E}v(t)-\mathbb{E}v^{*}(t),
$$
with the realization
\begin{equation}
\begin{aligned}
&\begin{cases}
d\mathbb{E}[x(t)]=\{(A_{11}(t)+\bar{A}_{11}(t))\mathbb{E}[x(t)]
+(B_{11}(t)\\+\bar{B_{11}}(t))\mathbb{E}[v(t)]\}dt, \\
\mathbb{E}[x(0)]=\mathbb{E}[x_{0}],
\end{cases}\\
&\begin{cases}
dx(t)-\mathbb{E}[x(t)]=
\{A_{11}(t)(x(t)-\mathbb{E}[x(t)])+\\
B_{11}(t)(v(t)-\mathbb{E}[v(t)])\}dt\\
+\{C_{11}(t)(x(t)-\mathbb{E}[x(t)])+(C_{11}(t)+\bar{C}_{11}(t))\mathbb{E}[x(t)]+\\
D_{11}(t)(v(t)-\mathbb{E}[v(t)])+(D_{11}(t)+\bar{D}_{11}(t))\mathbb{E}[v(t)]\}dW(t)\\
+\int_G\{E_{11}(t,\theta)(x(t-)-\mathbb{E}[x(t-)])+(E_{11}(t,\theta)+\\
\bar{E}_{11}(t,\theta))\mathbb{E}[x(t-)]+F_{11}(t,\theta)(v(t)-\mathbb{E}[v(t)])+\\
(F_{11}(t,\theta)+\bar{F}_{11}(t,\theta))\mathbb{E}[v(t)]\}\tilde{N}_p(d\theta,dt),\\
x(0)-\mathbb{E}[x(0)]=x_{0}-\mathbb{E}[x_{0}].
\end{cases}
\end{aligned}
\end{equation}
\begin{equation*}
\begin{aligned}
&\mathbb{E}v(t)-\mathbb{E}v^{*}(t)=\mathbb{E}v(t)-\Psi(t)\mathbb{E}x(t),\\
&v(t)-\mathbb{E}v(t)-(v^{*}(t)-\mathbb{E}v^{*}(t))=
\\&v(t)-\mathbb{E}v(t)-\Phi(t)(x(t)-\mathbb{E}x(t)).
\end{aligned}
\end{equation*}
Then this is a linear continuous bijection. By inverse mapping theorem, $\mathcal{L}^{-1}_{1},\tilde{\mathcal{L}}^{-1}_{1}$ exists and $\mathcal{L}^{-1}_{1},\tilde{\mathcal{L}}^{-1}_{1}$ is bounded, which is determined by
\begin{equation*}
\left\{\begin{aligned}
&d\mathbb{E}[x(t)]=(A_{11}(t)+\bar{A}_{11}(t)-(B_{11}(t)+\bar{B_{11}}(t))\Psi(t))\\
&\mathbb{E}[x(t)]+(B_{11}(t)+\bar{B_{11}}(t))(\mathbb{E}[v(t)]-\mathbb{E}[v^{*}(t)])dt ,\\
&\mathbb{E}[x(0)]=\mathbb{E}[\xi]
\end{aligned}\right.
\end{equation*}
with
$$
\mathbb{E}v(t)=\Psi(t)\mathbb{E}x(t)+(\mathbb{E}v(t)-\mathbb{E}v^{*}(t)),
$$
and
\begin{equation*}
\left\{
\begin{aligned}
&dx(t)-\mathbb{E}[x(t)]=\big\{(A_{11}(t)-B_{11}(t)\Psi(t))(x(t)-\mathbb{E}[x(t)])\\
&+B_{11}(t)(v(t)-\mathbb{E}[v(t)]-(v^{*}(t)-\mathbb{E}v^{*}(t)))\big\}dt \\
&+\big\{(C_{11}(t)-D_{11}(t)\Psi(t))(x(t)-\mathbb{E}[x(t)])\\
&+(C_{11}(t)+\bar{C}_{11}(t))\mathbb{E}[x(t)]\\
&+D_{11}(t)(v(t)-\mathbb{E}[v(t)]-(v^{*}(t)-\mathbb{E}v^{*}(t)))\\
&+(D_{11}(t)+\bar{D}_{11}(t))\mathbb{E}[v(t)]\big\}dW(t)\\
&+\int_G\big\{(E_{11}(t,\theta)-F_{11}(t,\theta)\Psi(t))(x(t-)-\mathbb{E}[x(t-)])\\
&+(E_{11}(t,\theta)+\bar{E}_{11}(t,\theta))\mathbb{E}[x(t-)]\\
&+F_{11}(t,\theta)(v(t)-\mathbb{E}[v(t)]-(v^{*}(t)-\mathbb{E}v^{*}(t)))\\
&+(F_{11}(t,\theta)-\bar{F}_{11}(t,\theta))\mathbb{E}[v(t)]\big\}\tilde{N}_p(d\theta,dt),\\
&x(0)-\mathbb{E}[x(0)]=\xi-\mathbb{E}[\xi]
\end{aligned}
\right.
\end{equation*}
with
\begin{multline}
v(t)-\mathbb{E}v(t)=\\
\Phi(t)(x(t)-\mathbb{E}x(t))+(v(t)-\mathbb{E}v(t)-(v^{*}(t)-\mathbb{E}v^{*}(t))).\nonumber
\end{multline}
Then there exists $\varepsilon > 0, \delta > 0$, such that
\begin{equation*}
\begin{aligned}
&J_{1}\left(0,v; 0, 0\right)\\
=&\mathbb{E} \int_{0}^{T}\big\{
\{(v-\mathbb{E}v)-\Phi(x-\mathbb{E}x)\}^{\prime}\Sigma_{0}(P)\\
&\{(v-\mathbb{E}v)-\Phi(x-\mathbb{E}x)\} +\\
&\{(\mathbb{E}v)-\Psi(\mathbb{E}x)\}^{\prime}\Sigma_{2}(P)\{(\mathbb{E}v)-\Psi(\mathbb{E}x)\}  \big\}dt\\
\geq &\delta \mathbb{E}\int_{0}^{T}\big\{((v-\mathbb{E}v)-\Phi(x-\mathbb{E}x))^{2}+(\mathbb{E}v-\Psi(\mathbb{E}x))^{2}\big\}dt\\
=&\delta\left\|\mathcal{L}_{1} (v(t)-\mathbb{E}v(t))\right\|_{[0, T]}^{2}+ \delta \|\tilde{\mathcal{L}}_{1} (\mathbb{E}v(t))\|_{[0, T]}^{2}\\
=&\delta\frac{1}{\|\mathcal{L}_{1}^{-1}\|}\left\|v(t)-\mathbb{E}v(t)\right\|_{[0, T]}^{2}+ \delta\frac{1}{\|\tilde{\mathcal{L}}_{1}^{-1}\|} \|\mathbb{E}v(t)\|_{[0, T]}^{2}\\
\geq &\varepsilon(\|v(t)-\mathbb{E}v(t)\|_{[0, T]}^{2}+\|\mathbb{E}v(t)\|_{[0, T]}^{2})\\
>&0,
\end{aligned}
\end{equation*}
which yields $\|\tilde{\mathcal{L}}\|<\gamma$. The sufficiency of MF-SJBRL is completely proved.

Necessity:
We study the global solvability of DREs. The function
\begin{equation*}
f\left(t,P\right)= \mathcal{S}(P)-\mathcal{G}(P)\Sigma_{0}^{-1}(P)\mathcal{G}^{\prime}(P) - \dot{P}
\end{equation*}
is continuously differentiable on $[0,T]\times D_{f}$, where $D_{f}=\left\{P: \operatorname{det}\left(\Sigma_{0}^{\gamma}(t, P(t))\right) \neq 0\right\}$.
The global solution of DREs is equivalent to the solution of
$$P(t)=P(T)+\int_{t}^{T}f\left(t,P\right)dt.$$
Define $\varphi(\hat{P})=-\Sigma_{0}(\hat{P})^{-1}\mathcal{G}^{\prime}(\hat{P})$, and
\begin{equation}\label{quasi}
F(t,P;\hat{P})= \left(\begin{array}{c}
                                   I \\
                                   \varphi(\hat{P})
                                 \end{array}\right)^{\prime}\left(
                                 \begin{array}{cc}
  \mathcal{S}(P)  & \mathcal{G}(P) \\
  \mathcal{G}^{\prime}(P)& \Sigma_{0}(P)
\end{array}
\right)
                                 \left(
                                 \begin{array}{c}
                                   I \\
                                   \varphi(\hat{P})
                                 \end{array}
                                 \right).
\end{equation}
Obviously,
\begin{equation*}
\begin{aligned}
&F(t,P;\hat{P})= \dot{P}+P (A_{11}+B_{11}\varphi(\hat{P}))+(A_{11}+B_{11}\varphi(\hat{P}))^{\prime}P\\
&+(C_{11}+D_{11}\varphi(\hat{P}))^{\prime} P (C_{11}+D_{11}\varphi(\hat{P})) \\
&+\int_G\{(E_{11}+F_{11}\varphi(\hat{P}))(\theta)^{\prime}P (E_{11}+F_{11}\varphi(\hat{P}))(\theta)\}\nu(d\theta)\\
& -M_{11}^{\prime} M_{11} +\gamma^{2}\varphi(\hat{P})^{\prime}\varphi(\hat{P}).
\end{aligned}
\end{equation*}
Then construct an iteration sequence below. At first,
let $\hat{P}=0$, then
\begin{equation*}
\begin{cases}
    F(t,P_{1};\hat{P})=0,\\
    P_{1}(T)=0.
\end{cases}
\end{equation*}
It is a linear ordinary differential equation which has a unique solution $P_{1}$. Next, let $\hat{P}=P_{1}$, then
\begin{equation*}
\begin{cases}
    F(t,P_{2};P_{1})=0,\\
    P_{2}(T)=0.
\end{cases}
\end{equation*}

Repeat the above step to obtain the sequence $\{P_{n}\}_{n=1}^{\infty}.$ And
\begin{equation*}
\begin{aligned}
&-d(P_{n}-P_{n+1})\\
=&(P_{n}-P_{n+1}) \tilde{A}_{n}+\tilde{A}_{n}^{\prime} (P_{n}-P_{n+1})+\tilde{C}_{n}^{\prime} (P_{n}-P_{n+1}) \tilde{C}_{n} \\
& +\int_G\{\tilde{E}_{n}(\theta)^{\prime}(P_{n}-P_{n+1}) \tilde{E}_{n}(\theta)\}\nu(d\theta)\\
&+P_{n}B_{11}(\varphi(P_{n-1})-\varphi(P_{n}))+(\varphi(P_{n-1})-\varphi(P_{n}))^{\prime}B_{11}^{\prime}P_{n}\\
&-\gamma^{2}\varphi(P_{n})^{\prime}\varphi(P_{n})+\gamma^{2}\varphi(P_{n-1})^{\prime}\varphi(P_{n-1})\\
&+\tilde{C}_{n-1}^{\prime} P_{n} \tilde{C}_{n-1}-\tilde{C}_{n}^{\prime} P_{n} \tilde{C}_{n}\\
&+\int_G\{\tilde{E}_{n-1}(\theta)^{\prime}P_{n} \tilde{E}_{n-1}(\theta)\}\nu(d\theta)\\
&-\int_G\{\tilde{E}_{n}(\theta)^{\prime}P_{n} \tilde{E}_{n}(\theta)\}\nu(d\theta)\\
=&(P_{n}-P_{n+1}) \tilde{A}_{n}+\tilde{A}_{n}^{\prime} (P_{n}-P_{n+1})+\tilde{C}_{n}^{\prime} (P_{n}-P_{n+1})\cdot \\
&\tilde{C}_{n} +\int_G\{\tilde{E}_{n}(\theta)^{\prime}(P_{n}-P_{n+1}) \tilde{E}_{n}(\theta)\}\nu(d\theta)\\
&+(\varphi(P_{n-1})-\varphi(P_{n}))^{\prime}
\Sigma_{0}(P_{n})(\varphi(P_{n-1})-\varphi(P_{n})),\\
\end{aligned}
\end{equation*}
where
$\tilde{A}_{n}=A_{11}+B_{11}\varphi(P_{n}),\tilde{C}_{n}=C_{11}+D_{11}\varphi(P_{n}),\tilde{E}_{n}=E_{11}+F_{11}\varphi(P_{n}).$
By Lemma \ref{lemmaPpositive}, $(P_{n}-P_{n+1})\geq 0.$
Repeat the same procedure, we can also get the decreasing sequence $\{Q_{n}\}_{n=1}^{\infty}$.

By Lemma \ref{lemmalinearj1} and Lemma \ref{lemmainitial}, when $\xi=x W, x\in\mathbb{R}^{n}$,
\begin{multline}
J_{1}\left(0,\varphi(x^{\varphi,\psi}-\mathbb{E}x^{\varphi,\psi})+ \psi\mathbb{E}x^{\varphi,\psi}; t, \xi\right)=\\
\mathbb{E}\left\langle xW, P_{n}xW\right\rangle\geq -\mu \mathbb{E}|xW|^{2}.\nonumber
\end{multline}
Then $P_{n}(t)\geq -\mu I$ for $t \in[0,T].$ When $\xi\in\mathbb{R}^{n}$, we have
\begin{multline}
J_{1}\left(0,\varphi(x^{\varphi,\psi}-\mathbb{E}x^{\varphi,\psi})+ \psi\mathbb{E}x^{\varphi,\psi}; t, \xi\right)\\
=\left\langle\mathbb{E}\xi, Q_{n}\mathbb{E}\xi\right\rangle\geq -\mu \mathbb{E}|\xi|^{2}.\nonumber
\end{multline}
Then $Q_{n}(t)\geq -\mu I$ for $t \in[0,T].$

Considering that $0\geq P_{1} \geq P_{2} \geq \ldots \geq P_{n} \geq \ldots \geq -\mu I$ and $0\geq Q_{1} \geq Q_{2} \geq \ldots \geq  Q_{n} \geq \ldots \geq -\mu I$. By monotone convergence theorem, there exists $P,Q$ such that $P_{n}\rightarrow P,$ $Q_{n}\rightarrow Q.$ Because of Lebesgue's dominated convergence theorem,
\begin{equation*}
\begin{aligned}
P(t)&=\lim_{n\rightarrow\infty}P_{n}(t)=P(T)+\lim_{n\rightarrow\infty}\int_{t}^{T}f\left(s,P_{n};P_{n-1}\right)ds\\
&=P(T)+\int_{t}^{T}\lim_{n\rightarrow\infty}f\left(s,P_{n};P_{n-1}\right)ds\\
&=P(T)+\int_{t}^{T}f\left(s,P;P\right)ds
\end{aligned}
\end{equation*}
satisfies (\ref{3.4}).
Moreover, by Lemma \ref{main},
\begin{equation*}\Sigma_{0}(t, P(t))=\lim_{n\rightarrow\infty} \Sigma_{0}(t, P_{n}(t))\geq \delta I > 0,
\end{equation*}
 so is $\Sigma_{2}^{\gamma}(t, P(t))$ on $[0,T]$.
Repeating the procedure for $Q_{n}$, and then derives the DRE (\ref{3.4})-(\ref{3.5}) having a solution $(P,Q)$ on $[0,T].$

Suppose $\widetilde{P} \in C\left([0, T] ; \mathbb{R}^{n \times n}\right)$ is another solution of (\ref{3.4}). Set $\widehat{P} \triangleq P-\widetilde{P}$. Then $\widehat{P}$ satisfies
\begin{equation*}
\left\{
\begin{aligned}
&\mathcal{S}(\widehat{P})+M_{11}^{\prime}M_{11}-\mathcal{G}^{\prime}(\widehat{P})\Sigma_{0}^{-1}(P) \mathcal{G}^{\prime}(P)\\
&-\mathcal{G}(\tilde{P})\Sigma_{0}^{-1}(\widetilde{P})\mathcal{G}^{\prime}(\widehat{P}) \\
&+\mathcal{G}(\widetilde{P}) \Sigma_{0}^{-1}(P) D_{11}^{\prime} \widehat{P} D_{11} \Sigma_{0}^{-1}(\widetilde{P})\mathcal{G}^{\prime}(\widetilde{P})+\mathcal{G}(\widetilde{P})\cdot\\
& \Sigma_{0}^{-1}(P) \int_{G}F_{11}^{\prime} \widehat{P} F_{11} \nu(d\theta) \Sigma_{0}^{-1}(\widetilde{P})\mathcal{G}^{\prime}(\widetilde{P})=0 ,\\
&\widehat{P}(T) =0,
\end{aligned}
\right.
\end{equation*}
where $\Sigma_{0}(P)>0$ and $\Sigma_{0}(\widetilde{P}) >0$. Since $\left|\Sigma_{0}^{-1}(P)\right|$ and $\left|\Sigma_{0}^{-1}(\widetilde{P})\right|$ are uniformly bounded due to their continuity, we can apply Gronwall's inequality to get $\widehat{P}(t) \equiv 0$. This proves the uniqueness of the equation (\ref{3.4}).
Repeating the previous steps, the uniqueness for equation (\ref{3.5}) is derived due to the uniform boundedness of all the coefficients. The proof is completed.
\end{proof}

\section*{References}
\bibliographystyle{IEEEtran}
\bibliography{reference}

@article{2004,
  title={Stochastic ${H}_{2}/{H}_{\infty}$ control with state-dependent noise},
  author={Chen, Bor-Sen and Zhang, Weihai},
  journal={IEEE Transactions on AutomaticControl},
  volume={49},
  number={4},
  pages={45--57},
  year={2004}
}

@article{2007MF,
  title={Mean field games},
  author={Lasry, Jean-Michel and Lions, Pierre-Louis},
  journal={Japanese Journal of Mathematics},
  volume={2},
  number={1},
  pages={229--260},
  year={2007},
  publisher={Springer}
}

@book{mfg,
  title={Probabilistic theory of mean field games with applications I-II},
  author={Carmona, Ren{\'e} and Delarue, Fran{\c{c}}ois and others},
  year={2018},
  publisher={Springer}
}

@article{2009,
  title={Mean-field backward stochastic differential equations: a limit approach},
  author={Buckdahn, Rainer and Djehiche, Boualem and Li, Juan and Peng, Shige},
  journal={The Annals of Probability},
  volume={37},
  number={4},
  pages={1524--1565},
  year={2009},
  publisher={Institute of Mathematical Statistics}
}

@article{2009PDE,
  title={Mean-field backward stochastic differential equations and related partial differential equations},
  author={Buckdahn, Rainer and Li, Juan and Peng, Shige},
  journal={Stochastic Processes and their Applications},
  volume={119},
  number={10},
  pages={3133--3154},
  year={2009},
  publisher={Elsevier}
}

@article{2011,
  title={A maximum principle for {SDEs} of mean-field type},
  author={Andersson, Daniel and Djehiche, Boualem},
  journal={Applied Mathematics \& Optimization},
  volume={63},
  number={3},
  pages={341--356},
  year={2011},
  publisher={Springer}
}

@article{2011mf,
  title={A general stochastic maximum principle for {SDEs} of mean-field type},
  author={Buckdahn, Rainer and Djehiche, Boualem and Li, Juan},
  journal={Applied Mathematics \& Optimization},
  volume={64},
  number={2},
  pages={197--216},
  year={2011},
  publisher={Springer}
}

@article{2013a,
  title={Control of {M}cKean--{V}lasov dynamics versus mean field games},
  author={Carmona, Ren{\'e} and Delarue, Fran{\c{c}}ois and Lachapelle, Aim{\'e}},
  journal={Mathematics and Financial Economics},
  volume={7},
  number={2},
  pages={131--166},
  year={2013},
  publisher={Springer}
}

@article{2013,
  title={Linear-quadratic optimal control problems for mean-field stochastic differential equations},
  author={Yong, Jiongmin},
  journal={SIAM Journal on Control and Optimization},
  volume={51},
  number={4},
  pages={2809--2838},
  year={2013},
  publisher={SIAM}
}

@article{2015,
  title={Finite horizon mean-field stochastic ${H}_{2}/{H}_{\infty}$ control for continuous-time systems with $(x, v)$-dependent noise},
  author={Ma, Limin and Zhang, Tianliang and Zhang, Weihai and Chen, Bor-Sen},
  journal={Journal of the Franklin Institute},
  volume={352},
  number={12},
  pages={5393--5414},
  year={2015},
  publisher={Elsevier}
}

@book{jump,
  title={Financial modelling with jump processes},
  author={Tankov, Peter},
  year={2003},
  publisher={Chapman and Hall/CRC}
}

@article{2019,
  title={Linear-quadratic optimal control problems for mean-field stochastic differential equations with jumps},
  author={Tang, Maoning and Meng, Qingxin},
  journal={Asian Journal of Control},
  volume={21},
  number={2},
  pages={809--823},
  year={2019},
  publisher={Wiley Online Library}
}

@article{1992,
  title={McKean-{V}lasov \mbox{It\^o}-{S}korohod equations, and nonlinear diffusions with discrete jump sets},
  author={Graham, Carl},
  journal={Stochastic Processes and their Applications},
  volume={40},
  number={1},
  pages={69--82},
  year={1992},
  publisher={Elsevier}
}

@article{2021,
  title={${H}_{2}/{H}_{\infty}$ control for stochastic jump-diffusion systems with {M}arkovian switching},
  author={Wang, Meijiao and Meng, Qingxin and Shen, Yang},
  journal={Journal of Systems Science and Complexity},
  volume={34},
  number={3},
  pages={924--954},
  year={2021},
  publisher={Springer}
}

@article{wang_2023,
author={Wang, Meijiao and Meng, Qingxin and Shen, Yang and Shi, Peng},
title={Stochastic ${H}_{2}/{H}_{\infty }$ Control for Mean-Field Stochastic Differential Systems with $(x, u, v)$-Dependent Noise},
journal={Journal of Optimization Theory and Applications},
pages={1024--1060},
volume={197},
number={3},
ISSN={1573-2878},
year={2023}
}

@article{stackelberg,
  author = {Zhang, Suya and Zhang, Weihai and Meng, Qingxin},
  title = {Stackelberg Game Approach to Mixed Stochastic ${H}_{2}/{H}_{\infty }$ Control for Mean-Field Jump-Diffusions Systems},
  journal = {Applied Mathematics \& Optimization},
  volume={89},
  number={6},
year = {2023}
}

\end{document}